\documentclass[11pt,a4j]{article}

\usepackage{lineno,hyperref}
 \usepackage{amsmath,amssymb,amsfonts}
 \usepackage{amsthm}
 \usepackage{latexsym}
 \usepackage{array}
 \usepackage{color}

\theoremstyle{definition}
\newtheorem{theorem}{Theorem}[section]

\newtheorem{prop}[theorem]{Proposition}
\newtheorem{lemma}[theorem]{Lemma}
\newtheorem{rem}{Remark}
\newtheorem{example}{Example}
\newtheorem{defn}{Definition}

\DeclareMathOperator{\graph}{graph}

\allowdisplaybreaks[4] 

\title{Dirac pairs on Jacobi algebroids}
\author{Tomoya Nakamura  \thanks{Academic Support Center, Kogakuin University, \textup{2665-1}, Nakano-cho, Hachioji-shi, Tokyo, Japan}
 \\ email: \href{mailto:kt13676@ns.kogakuin.ac.jp}{kt13676@ns.kogakuin.ac.jp}}
\date{\today}

\begin{document}

  \maketitle
  
  \begin{abstract}
  We define Dirac pairs on Jacobi algebroids, which is a generalization of Dirac pairs on Lie algebroids introduced by Kosmann-Schwarzbach. We show the relationship between Dirac pairs on Lie and on Jacobi algebroids, and that Dirac pairs on Jacobi algebroids characterize several compatible structures on manifolds or on Jacobi algebroids.
  \end{abstract}

\section{Introduction}
Poisson and symplectic structures on smooth manifolds have wide application in the theory of integrable systems on smooth manifolds, especially even dimensional manifolds. These structures are generalized on Lie algebroids. One of more generalizations of Poisson 
structures on Lie algebroids are Dirac structures, which are defined on Lie bialgebroids in general \cite{LWX}. Dirac structures on a Lie algebroid $A$ are defined by using the Lie bialgebroid canonically determined for $A$. In terms of applications in the theory of integrable systems, compatible two structures, for example, $P\Omega$- and $\Omega N$-structures \cite{MM2}, are often used. The notion dealing with these compatible structures in a unified way is a Dirac pair, which was introduced by Kosmann-Schwarzbach \cite{K}.

On the other hand, contact structures can be defined on odd dimensional manifolds, and Jacobi structures are generalizations of contact structures. Jacobi structures are generalized as structures on Jacobi algebroids. As a generalization of both Jacobi structures on Jacobi algebroids and Dirac structures on Lie bialgebroids, we can define Dirac structures on Jacobi bialgebroids \cite{NC}. As in the case of Lie algebroid, Dirac structures on Jacobi algebroids can also be defined naturally. In addition, we can define several compatible structures on Jacobi algebroids, for example, $J\Omega$- and $\Omega N$-structures. In this paper, we define Dirac pairs on Jacobi bialgebroids and prove that $J\Omega$- and $\Omega N$-structures can be characterized by Dirac pairs. Furthermore, we investigate relationships between Dirac pairs on Lie and Jacobi bialgebroids.

This paper is divided into four sections. In Section \ref{section:Pleriminaries}, we recall several definitions, properties and examples of Lie and Jacobi algebroids, relations and Dirac pairs on Lie bialgebroids. In Section \ref{Dirac pairs on Jacobi bialgebroids}, we define Dirac pairs on Jacobi bialgebroids. We show that $(\overline{\graph\pi_{1}^{\sharp}},$ $\overline{\graph\pi_{2}^{\sharp}})$, $(\overline{\graph\pi_{1}^{\sharp}},$ $\graph\omega_{2}^{\flat})$ and $(\graph\omega_{1}^{\flat},$ $\graph\omega_{2}^{\flat})$ are Dirac pairs on a Jacobi bialgebroid $((A,\phi_{0}),(A^{*},X_{0}))$ over $M$ if and only if $(\overline{\graph\tilde{\pi}_{1}^{\sharp}},\overline{\graph\tilde{\pi}_{2}^{\sharp}})$, $(\overline{\graph\tilde{\pi}_{1}^{\sharp}},$ $\graph\tilde{\omega}_{2}^{\flat})$ and $(\graph\tilde{\omega}_{1}^{\flat},\graph\tilde{\omega}_{2}^{\flat})$ are Dirac pairs on the induced Lie bialgebroid $(\tilde{A}_{\bar{\phi}_{0}},\tilde{A}^{*}_{\hat{X}_{0}})$ over $M\times \mathbb{R}$, respectively. Here $\pi_{i}\in\Gamma(\Lambda^2A)$ and $\omega_{i}\in\Gamma(\Lambda^2A^*)$ are elements satisfying the Maurer-Cartan type equation ($i=1,2$), and we set $\tilde{\pi}_{i}:=e^{-t}\pi_{i}$ in $\Gamma (\Lambda ^{2}\tilde{A})$ and $\tilde{\omega}_{i}:=e^{t}\omega_{i}$ in $\Gamma (\Lambda ^{2}\tilde{A}^{*})$, where $t$ is the standard coordinate in $\mathbb{R}$. This is the main theorem in this paper. In Section \ref{Dirac pairs on Jacobi algebroids}, we investigate Dirac pairs on Jacobi algebroids. We introduce Jacobi pairs and $\phi_0$-presymplectic pairs defined by using Dirac pairs on Jacobi algebroids. We show the relationship between Jacobi (resp. $\phi_0$-presymplectic) pairs and Poisson (resp. presymplectic) pairs, and prove
that there exist a one-to-one correspondence between the non-degenerate 
Jacobi pairs and the $\phi_0$-symplectic pairs on Jacobi algebroids. Moreover, we introduce J$\Omega$- and $\Omega $N-structures on Jacobi algebroids. We show the relationship between J$\Omega$-(resp. $\Omega $N-)structures on Jacobi algebroids and P$\Omega$-(resp. $\Omega$N-)structures on Lie algebroids, and prove that J$\Omega$- and $\Omega $N-structures can be characterized by Dirac pairs on Jacobi algebroids.

\section{Preliminaries}\label{section:Pleriminaries}

 \subsection{Lie and Jacobi algebroids}\label{Lie, Jacobi and generalized Courant algebroids}

A {\it Lie algebroid} over a manifold $M$ is a vector bundle
$A\rightarrow M$ equipped with a Lie bracket $[\cdot,\cdot]_{A}$ on $\Gamma (A)$ and a bundle map $\rho_{A}:A\rightarrow TM$ over $M$, called the {\it anchor}, satisfying the following condition: for any $X,Y$ in $\Gamma (A)$ and $f$ in $C^{\infty }(M)$,
 \begin{align*}
[X,fY]_{A}=f[X,Y]_{A}+(\rho_{A}(X)f)Y.
 \end{align*}

The {\it Schouten bracket} on $\Gamma (\Lambda ^*A)$ is defined similarly to the Schouten bracket $[\cdot ,\cdot ]$ on $\mathfrak{X}^*(M)$. That is, the Schouten bracket $[\cdot ,\cdot ]_A:\Gamma (\Lambda ^kA)\times \Gamma (\Lambda ^lA)\rightarrow \Gamma (\Lambda ^{k+l-1}A)$ is defined as the unique extension of the Lie bracket $[\cdot ,\cdot ]_A$ on $\Gamma (A)$ such that
\begin{align*}
&[f,g]_A=0;\\
&[X,f]_A=\rho_A(X)f;\\
&[X,Y]_A\ \mbox{is the Lie bracket on}\ \Gamma (A);\\
&[D_1,D_2\wedge D_3]=[D_1,D_2]\wedge D_3+(-1)^{\left(a_1+1\right)a_2}D_2\wedge [D_1, D_3];\\
&[D_1,D_2]_A=-(-1)^{(a_1-1)(a_2-1)}[D_2,D_1]_A
\end{align*}
for any $f,g$ in $C^\infty(M)$, $X,Y$ in $\Gamma (A)$ and $D_i$ in $\Gamma (\Lambda ^{a_{i}}A)$. The {\it differential} of the Lie algebroid $A$ is an operator $d_A:\Gamma (\Lambda ^k A^*)\rightarrow \Gamma (\Lambda ^{k+1} A^*)$ defined by for any $\omega $ in $\Gamma (\Lambda^k A^*)$ and $X_0,\dots ,X_k$ in $\Gamma(A)$,
\begin{align}\label{A-gaibibun}
(d_A \omega)(X_0, \dots ,X_k)&=\sum_{i=0}^k(-1)^i\rho_A (X_i)(\omega(X_0,\dots ,\hat{X_i},\dots ,X_k))\nonumber \\
&\quad+\sum _{i<j}^{}(-1)^{i+j}\omega ([X_i,X_j]_A , X_0,\dots ,\hat{X_i},\dots ,\hat{X_j},\dots ,X_k).
\end{align}
For any $X$ in $\Gamma (A)$, the {\it Lie derivative\index{Lie derivative}} $\mathcal{L}_X^A:\Gamma (\Lambda ^k A^*)\rightarrow \Gamma (\Lambda ^k A^*)$ is defined by the {\it Cartan formula} $\mathcal{L}_X^A:=d_A \iota _X +\iota _X d_A$ and $\mathcal{L}_X^A$ are extended on $\Gamma (\Lambda ^*A)$ in the same way as the usual Lie derivative $\mathcal{L}_X$ respectively. Then it follows that $\mathcal{L}_X^AD=[X,D]_A$ for any $D$ in $\Gamma (\Lambda^{*}A)$. 

\begin{example}
(i) Any finite dimensional real Lie algebra is a Lie algebroid over a point.

(ii) For any manifold $M$, the tangent bundle $(TM,[\cdot,\cdot],\mbox{id}_{TM})$ is a Lie algebroid over $M$, where $[\cdot,\cdot]$ is the usual Lie bracket on the vector fields $\mathfrak{X}(M)$. 

(iii) For any vector bundle $A$ over $M$, we set $[\cdot,\cdot]_{A}:=0$ and $\rho_{A}:=0$. Then $A_{0}:=(A,[\cdot,\cdot]_{A},\rho_{A})$ is a Lie algebroid. We call $([\cdot,\cdot]_{A},\rho_{A})$ the {\it trivial Lie algebroid structure} on $A$.
\end{example}

\begin{example}\label{A oplus R}
Let $A$ be a vector bundle over a manifold $M$ and set $A\oplus \mathbb{R}:=A\oplus(M\times \mathbb{R})$. Then the sections $\Gamma (\Lambda^k(A\oplus \mathbb{R}))$ and $\Gamma (\Lambda^k(A\oplus \mathbb{R})^*)$ can be identified with $\Gamma (\Lambda^kA)\oplus \Gamma (\Lambda^{k-1}A)$ and $\Gamma (\Lambda^kA^*)\oplus \Gamma (\Lambda^{k-1}A^*)$ as follows:
\begin{align}
(P,Q)((\alpha_1,f_1),&\dots,(\alpha_k,f_k))\nonumber\\
                     &=P(\alpha_1,\dots,\alpha_k)+\sum_i(-1)^{i+1}f_iQ(\alpha_1,\dots,\hat{\alpha}_i,\dots,\alpha_k),\label{multi-vector field formula}\\
(\alpha,\beta)((X_1,f_1),&\dots,(X_k,f_k))\nonumber\\
                         &=\alpha(X_1,\dots,X_k)+\sum_i(-1)^{i+1}f_i\beta(X_1,\dots,\hat{X}_i,\dots,X_k) \label{differential form formula}
\end{align}
for any $(P,Q)$ in $\Gamma (\Lambda^kA)\oplus \Gamma (\Lambda^{k-1}A)$, $(\alpha,\beta)$ in $\Gamma (\Lambda^kA^*)\oplus \Gamma (\Lambda^{k-1}A^*)$, $(\alpha _i,f_i)$ in $\Gamma (A^*)\oplus C^\infty(M)$ and $(X_i,f_i)$ in $\Gamma (A)\oplus C^\infty(M)$. Moreover under the identifications, the exterior products are given by
\begin{align*}
(P_1,Q_1)\wedge(P_2,Q_2)&=(P_1\wedge P_2,Q_1\wedge P_2+(-1)^{a_1}P_1\wedge Q_2),\\
(\alpha_1,\beta_1)\wedge(\alpha_2,\beta_2)&=(\alpha_1\wedge \alpha_2,\beta_1\wedge \alpha_2+(-1)^{a_1}\alpha_1\wedge \beta_2)
\end{align*}
for any $(P_i,Q_i)$ in $\Gamma (\Lambda^{a_i}A)\oplus \Gamma (\Lambda^{a_i-1}A)$ and $(\alpha_i,\beta_i)$ in $\Gamma (\Lambda^{a_i}A^*)\oplus \Gamma (\Lambda^{a_i-1}A^*)$. Now, assume that $A$ is a Lie algebroid over $M$. Then $(A\oplus \mathbb{R},[\cdot,\cdot]_{A\oplus\mathbb{R}},\rho_{A}\circ\mbox{pr}_1)$ is also a Lie algebroid over $M$, where the bracket $[\cdot,\cdot]_{A\oplus\mathbb{R}}$ is defined by
\begin{align}\label{A oplus R no kakko}
[(X,f),(Y,g)]_{A\oplus\mathbb{R}}&:=([X,Y]_{A},\rho_{A}(X)g-\rho_{A}(Y)f)
\end{align}
and the map $\mbox{pr}_1:A\oplus \mathbb{R}\rightarrow A$ is the canonical projection to the first factor. In this case, the defferential $d_{A\oplus \mathbb{R}}$ of the Lie algebroid $A\oplus \mathbb{R}$ and the Schouten bracket $[\cdot,\cdot]_{A\oplus \mathbb{R}}$ are given by
\begin{align*}
d_{A\oplus \mathbb{R}}(\alpha,\beta)&=(d_{A}\alpha,-d_{A}\beta),\\
[(P_1,Q_1),(P_2,Q_2)]_{A\oplus\mathbb{R}}&=([P_1,P_2],(-1)^{k+1}[P_1,Q_2]-[Q_1,P_2])
\end{align*}
for any $(\alpha,\beta)$ in $\Gamma (\Lambda^{k}A^{*})\oplus \Gamma (\Lambda^{k-1}A^{*})$ and $(P_i,Q_i)$ in $\Gamma (\Lambda^{k}A)\oplus \Gamma (\Lambda^{k}A)$. 
\end{example}

A pair $(A,\phi_0)$ is a {\it Jacobi algebroid} over $M$ if $A=(A,[\cdot ,\cdot ]_A,\rho_A)$ is a Lie algebroid over $M$ and $\phi_0$ in $\Gamma(A^*)$ is $d_A$-closed, that is, $d_A\phi_{0}=0$. 

\begin{example}
For a Lie algebroid $A\oplus \mathbb{R}$ in Example \ref{A oplus R}, We set $\phi_0:=(0,1)$ in $\Gamma(A^{*}\oplus \mathbb{R})=\Gamma (A^{*})\oplus C^{\infty}(M)$. Then $(A\oplus \mathbb{R},\phi_0)$ is a Jacobi algebroid.
\end{example}

\begin{example}\label{trivial example}
For any Lie algebroid $A$ over $M$, we set $\phi_0:=0$. Then $(A,\phi_0)$ is a Jacobi algebroid. We call $\phi_0$ the {\it trivial Jacobi algebroid structure} on $A$. Therefore any Lie algebroid is a Jacobi algebroid.
\end{example}

For a Jacobi algebroid $(A,\phi_{0})$, there is the {\it $\phi_0$-Schouten bracket} $[\cdot,\cdot]_{A,\phi_0}$ on $\Gamma (\Lambda^*A)$ given by
\begin{align*}\label{phi-bracket}
[D_1,D_2]_{A,\phi_0}:=[D_1,D_2]_A+(a_1-1)&D_1\wedge \iota _{\phi_0}D_2\\
                                                     &-(-1)^{a_1+1}(a_2-1)\iota_{\phi_0}D_1\wedge D_2
\end{align*}
for any $D_i$ in $\Gamma(\Lambda ^{a_i}A)$, where $[\cdot,\cdot]_A$ is the Schouten bracket of the Lie algebroid $A$. The {\it $\phi_0$-differential} $d_{A,\phi_0}$ and the {\it $\phi_0$-Lie derivative} $\mathcal{L}_X^{A,\phi_0}$ are defined by
\begin{align*}
d_{A,\phi_0}\omega :=d_A\omega +\phi_0\wedge \omega,\quad \mathcal{L}_X^{A,\phi_0}:=\iota _X\circ d_{A,\phi_0}+d_{A,\phi_0}\circ \iota_X
\end{align*}
for any $\omega$ in $\Gamma (\Lambda ^*A^*)$ and $X$ in $\Gamma (A)$.
 
 We notice that
\begin{align*}
(d_{A,\phi_0}\omega )(X_0,&\dots,X_k)\\
                          &:=\sum_i(-1)^{i+1}\rho_{A,\phi_0}(X_i)\omega(X_0,\dots,\hat{X}_i,\dots,X_k)\\
                          &\quad +\sum_{i<j}(-1)^{i+j}\omega([X_i,X_j]_A,X_0,\dots,\hat{X}_i,\dots,\hat{X}_j,\dots,X_k)
\end{align*}
for any $\omega$ in $\Gamma (\Lambda^kA^*)$ and $X_i$ in $\Gamma (A)$, and that
\begin{align*}
\mathcal{L}_X^{A,\phi_0}\omega=\mathcal{L}_X^A\omega+\phi_0(X)\omega
\end{align*}
for any $\omega$ in $\Gamma (\Lambda ^*A^*)$ and $X$ in $\Gamma (A)$. Here $\rho_{A,\phi_0}(X)f:=\rho_A(X)f+\langle \phi_0,X\rangle f$ for any $X$ in $\Gamma (A)$ and $f$ in $C^\infty(M)$. We call a $d_{A,\phi_0}$-closed $2$-cosection $\omega$, i.e., $d_{A,\phi_0}\omega=0$, a {\it $\phi_0$-presymplectic structure} on $(A,\phi_{0})$. A $\phi_0$-presymplectic structure $\omega$ is called a {\it $\phi_0$-symplectic structure} if $\omega$ is non-degenerate.

 \begin{example}\label{contact example}
We consider $A:=TM\oplus\mathbb{R}$ and $\phi_{0}:=(0,1)$ in $\Omega^{1}(M)\oplus C^{\infty}(M)$. Then any $\omega $ in $\Omega^{2}(M)\oplus \Omega^{1}(M)$ can be written as $\omega=(\alpha,\beta)\ (\alpha\in\Omega^{2}(M),\beta\in\Omega^{1}(M))$. Since
 \begin{align*}
d_{A,\phi_{0}}\omega=d_{TM\oplus\mathbb{R},(0,1)}(\alpha,\beta)=(d\alpha,\alpha-d\beta),
\end{align*}
$\omega$ is $(0,1)$-presymplectic on $(TM\oplus\mathbb{R},(0,1))$ if and only if $\omega =(d\beta,\beta)\ (\beta\in\Omega^{1}(M))$. Moreover setting $\dim M=2n+1$, we see that a $(0,1)$-presymplectic strucutre $\omega $ is non-degenerate if and only if $\beta\wedge (d\beta)^{n}\neq 0$, that is, $\beta$ is a {\it contact structure} on $M$. Therefore a $(0,1)$-symplectic structure on $(TM\oplus \mathbb{R},(0,1))$ is just a contact structure on $M$.
 \end{example}

A {\it Jacobi structure} on a Jacobi algebroid $(A,\phi_{0})$ is a $2$-section $\pi $ in $\Gamma (\Lambda^{2}A)$ satisfying the condition
\begin{equation}\label{Jacobi def equation}
[\pi,\pi]_{A,\phi_{0}}=0.
\end{equation}
For any $2$-section $\pi $ in $\Gamma (\Lambda^2A)$, we define a skew-symmetric bilinear bracket $[\cdot,\cdot]_{\pi,\phi_{0}} $ on $\Gamma (A^*)$ by for any $\xi,\eta$ in $\Gamma (A^*)$,
\begin{align}
[\xi,\eta]_{\pi,\phi_{0}}:=\mathcal{L}_{\pi^\sharp \xi}^{A,\phi_0}\eta -\mathcal{L}_{\pi^\sharp \eta}^{A,\phi_0}\xi -d_{A,\phi_0}\langle \pi^\sharp \xi,\eta \rangle,
\end{align}
where a bundle map $\pi^\sharp:A^{*}\rightarrow A$ over $M$ is defined by $\langle \pi^\sharp \xi,\eta\rangle:=\pi(\xi,\eta)$. Then the following holds:
 \begin{align}\label{Jacobi structure bracket property}
\frac{1}{2}[\pi ,\pi ]_{A,\phi_0}(\xi,\eta,\cdot)=[\pi^\sharp \xi,\pi^\sharp \eta ]_{A}-\pi^\sharp[\xi,\eta]_{\pi,\phi_{0}}.
\end{align}
In particular, a $2$-section $\pi $ is a Jacobi structure if and only if
 \begin{align}
[\pi^\sharp \xi,\pi^\sharp \eta ]_{A,\phi_0}=\pi^\sharp[\xi,\eta]_{\pi,\phi_{0}}.
\end{align}
Two Jacobi structures $\pi$ and $\pi'$ on a Jacobi algebroid $(A,\phi_0)$ is {\it compatible} if $\pi+\pi'$ is also a Jacobi structure. Obviously, two Jacobi structures $\pi$ and $\pi'$ are compatible if and only if $[\pi,\pi']_{A,\phi_0}=0$. By (\ref{Jacobi def equation}) and (\ref{Jacobi structure bracket property}), we obtain
 \begin{align}\label{Jacobi structure no wa}
[\pi ,\pi' ]_{A,\phi_0}(\xi,\eta,\cdot)=[\pi^\sharp \xi,\pi^{\prime\sharp} \eta ]_{A}+[\pi^{\prime\sharp} \xi,\pi^\sharp \eta ]_{A}-\pi^\sharp[\xi,\eta]_{\pi',\phi_{0}}-\pi^{\prime\sharp}[\xi,\eta]_{\pi,\phi_{0}}
\end{align}
for two Jacobi structures $\pi$ and $\pi'$.

\begin{example}[Poisson structures]
For any Lie algebroid $A$ equipped with the trivial Jacobi algebroid structure $0$, it follows that $[\cdot,\cdot]_{A,0}=[\cdot,\cdot]_{A}$. Hence Jacobi structures on $(A,0)$ are just Poisson structures on $A$
. Properties of the bracket $[\cdot,\cdot]_{\pi}:=[\cdot,\cdot]_{\pi,0}$, where $\pi $ is any $2$-section on $A$, is reproduced by (\ref{Jacobi structure bracket property}) and so on.
\end{example}

\begin{example}\label{Jacobi manifold}
Let $A$ be a Lie algebroid over $M$, $\Lambda $ a $2$-section on $A$ and $E$ a section on $A$ satisfying
\begin{align*}
[\Lambda ,\Lambda ]_{A}=2E\wedge \Lambda ,\quad [E,\Lambda ]_{A}=0.
\end{align*}
Then a pair $(\Lambda ,E)$ in $\Gamma (\Lambda ^{2}A)\oplus \Gamma (A)\cong \Gamma (\Lambda ^{2}(A\oplus \mathbb{R}))$ is a Jacobi structure on a Jacobi algebroid $(A\oplus \mathbb{R},(0,1))$, i.e., it satisfies $[(\Lambda ,E),(\Lambda ,E)]_{A\oplus \mathbb{R},(0,1)}=0$. When $(\Lambda ,E)$ is a Jacobi structure on $(TM\oplus \mathbb{R},(0,1))$, we call it a {\it Jacobi structure on $M$} and a triple $(M,\Lambda,E)$ a {\it Jacobi manifold}. If $\pi $ is a Poisson structure on $A$, Then $(\pi,0)$ is a Jacobi structure on $(A\oplus \mathbb{R},(0,1))$.
\end{example}

It is well known that there exists a one-to-one correspondence between $\phi_{0}$-symplectic structures on $(A,\phi_{0})$ and non-degenerate Jacobi structures on $(A,\phi_{0})$. In fact, for a non-degenerate Jacobi structure $\pi$ on $(A,\phi_{0})$, a $2$-cosection $\omega_{\pi}$ characterized by $\omega_{\pi}^{\flat}=-(\pi^{\sharp})^{-1}$ is $\phi_{0}$-symplectic on $(A,\phi_{0})$, where for any $2$-cosection $\omega$, a bundle map $\omega^\flat:A\rightarrow A^{*}$ over $M$ is defined by $\langle \omega^\flat X,Y\rangle:=\omega(X,Y)$.

Let $(A,\phi_0)$ be a Jacobi algebroid over $M$. We set $\tilde{A}:=A\times \mathbb{R}$. Then $\tilde{A}$ is a vector bundle over $M\times \mathbb{R}$. The sections $\Gamma (\tilde{A})$ can be identified with the set of time-dependent sections of $A$. Under this identification, we can define two Lie algebroid structures $([\cdot,\cdot\hat{]}_A^{\phi_0},\hat{\rho}_A^{\phi_0})$ and $([\cdot,\cdot\bar{]}_A^{\phi_0},\bar{\rho}_A^{\phi_0})$ on $\tilde{A}$, where for any $\tilde{X}$ and $\tilde{Y}$ in $\Gamma(\tilde{A})$,
\begin{align}
\label{hat kakko} [\tilde{X},\tilde{Y}\hat{]}_A^{\phi_0}&:=e^{-t}\left([\tilde{X},\tilde{Y}]_A+\langle\phi_0,\tilde{X}\rangle\left(\frac{\partial \tilde{Y}}{\partial t}-\tilde{Y}\right)-\langle\phi_0,\tilde{Y}\rangle\left(\frac{\partial \tilde{X}}{\partial t}-\tilde{X}\right)\right),\\ 
\label{hat anchor} \hat{\rho}_A^{\phi_0}(\tilde{X})&:=e^{-t}\left(\rho_A(\tilde{X})+\langle \phi_0,\tilde{X}\rangle \frac{\partial}{\partial t}\right),\\
\label{bar kakko} [\tilde{X},\tilde{Y}\bar{]}_A^{\phi_0}&:=[\tilde{X},\tilde{Y}]_A+\langle\phi_0,\tilde{X}\rangle\frac{\partial \tilde{Y}}{\partial t}-\langle\phi_0,\tilde{Y}\rangle\frac{\partial \tilde{X}}{\partial t},\\ 
\label{bar anchor} \bar{\rho}_A^{\phi_0}(\tilde{X})&:=\rho_A(\tilde{X})+\langle \phi_0,\tilde{X}\rangle \frac{\partial}{\partial t}.
\end{align}
Conversely, for a Lie algebroid $A$ over $M$ and a section $\phi_0$ on $A$, if the triple $(\tilde{A},[\cdot,\cdot\hat{]}_A^{\phi_0},\hat{\rho}_A^{\phi_0})$ (resp. $(\tilde{A},[\cdot,\cdot\bar{]}_A^{\phi_0},\bar{\rho}_A^{\phi_0})$) defined by (\ref{hat kakko}) and (\ref{hat anchor}) (resp. (\ref{bar kakko}) and (\ref{bar anchor})) is a Lie algebroid over $M\times \mathbb{R}$, then $(A,\phi_0)$ is a Jacobi algebroid over $M$, i.e., $d_{A}\phi_0=0$. A vector bundle $\tilde{A}$ equipped with the Lie algebroid structure $([\cdot,\cdot\hat{]}_A^{\phi_0},\hat{\rho}_A^{\phi_0})$ (resp. $([\cdot,\cdot\bar{]}_A^{\phi_0},\bar{\rho}_A^{\phi_0})$) is denoted by $\tilde{A}_{\hat{\phi}_{0}}$ (resp. $\tilde{A}_{\bar{\phi}_{0}}$). Let $\hat{d}_{A}^{\phi_{0}}$ (resp. $\bar{d}_{A}^{\phi_{0}}$) and $\widehat{\mathcal{L}^{A}}^{\phi_{0}}$ (resp. $\overline{\mathcal{L}^{A}}^{\phi_{0}}$) be the differential of $\tilde{A}_{\hat{\phi}_{0}}$ (resp. $\tilde{A}_{\bar{\phi}_{0}}$) and the Lie derivative on $\tilde{A}_{\hat{\phi}_{0}}$ (resp. $\tilde{A}_{\bar{\phi}_{0}}$), respectively. Then for any $\tilde{f}$ in $C^{\infty}(M\times \mathbb{R})$ and $\tilde{\phi}$ in $\Gamma (\tilde{A})$, the following formulas hold \cite{IM}:
\begin{align}
&\hat{d}_{A}^{\phi_{0}}\tilde{f}=e^{-t}\left(d_{A}\tilde{f}+\frac{\partial \tilde{f}}{\partial t}\phi_{0}\right),\quad \hat{d}_{A}^{\phi_{0}}\tilde{\phi}=e^{-t}\left(d_{A,\phi_{0}}\tilde{\phi}+\phi_{0}\wedge\frac{\partial \tilde{\phi}}{\partial t}\right);\\
&\bar{d}_{A}^{\phi_{0}}\tilde{f}=d_{A}\tilde{f}+\frac{\partial \tilde{f}}{\partial t}\phi_{0},\quad \bar{d}_{A}^{\phi_{0}}\tilde{\phi}=d_{A}\tilde{\phi}+\phi_{0}\wedge\frac{\partial \tilde{\phi}}{\partial t}.
\end{align}

The definition and properties of Jacobi bialgebroids 
are the followings.

 \begin{defn}[\cite{IM}]\label{Jacobi bialgebroid def}
Let $(A,\phi_0)$ and $(A^*,X_0)$ be two Jacobi algebroids over $M$ in duality. Then a pair $((A,\phi_0),(A^*,X_0))$ is a {\it Jacobi bialgebroid} over $M$ if for any $X,Y$ in $\Gamma(A)$ and $P$ in $\Gamma (\Lambda ^kA)$,
\begin{align*}
 d_{A^{*}\!,X_0}[X,Y]_{A}&=[d_{A^{*}\!,X_0}X,Y]_{A,\phi_0}+[X,d_{A^{*}\!,X_0}Y]_{A,\phi_0},\\
\mathcal{L}^{A,\phi_0}_{X_0}P&+\mathcal{L}^{A^*\!,X_0}_{\phi_0}P=0,
\end{align*}
 where $d_{A^{*}\!,X_0}$ is the $X_0$-differential and $\mathcal{L}^{A^*\!,X_0}$ is the $X_0$-Lie derivative of $(A^*,X_0)$.
 \end{defn}

  \begin{example}[Lie bialgebroids \cite{MX}]\label{Lie bialgebroids}
Let $A$ and $A^{*}$ be Lie algebroids in duality equipped with the trivial Jacobi algebroid structures $0$. Then a pair $((A,0),(A^{*},0))$ is a Jacobi bialgebroid if and only if a pair $(A,A^{*})$ is a Lie bialgebroid.
 \end{example}

  \begin{example}[\cite{IM}]\label{trivial example}
 For any Jacobi algebroid $(A,\phi_{0})$ and its dual bundle $(A_{0}^{*},0)$ equipped with the trivial Lie and Jacobi algebroid structure, a pair $((A,\phi_0),(A_{0}^{*},0))$ is a Jacobi bialgebroid.
 \end{example}
 
 Proposition \ref{Jacobi and Lie bilagebroid} is the relation between a Jacobi and Lie bialgebroid. 
 
 \begin{prop}[\cite{IM}]\label{Jacobi and Lie bilagebroid}
A pair $((A,\phi_0),(A^*,X_0))$ is a Jacobi bialgebroid over $M$ if and only if a pair $\left(\tilde{A}_{\bar{\phi}_{0}},\tilde{A}^{*}_{\hat{X}_{0}}\right)=((\tilde{A},[\cdot,\cdot\bar{]}_A^{\phi_{0}},\bar{\rho}_A^{\phi_0}),(\tilde{A}^*,[\cdot,\cdot\hat{]}_A^{X_0},\hat{\rho}_A^{X_{0}}))$ is a Lie bialgebroid over $M\times \mathbb{R}$.
 \end{prop}
 
Proposition \ref{Jacobi bilagebroid gyaku ok} follows immediately by Proposition \ref{Jacobi and Lie bilagebroid}.

 \begin{prop}[\cite{IM}]\label{Jacobi bilagebroid gyaku ok}
 If $((A,\phi_0),(A^*,X_0))$ is a Jacobi bialgebroid, then so is $((A^*,X_0),(A,\phi_0))$.
 \end{prop}

To define Dirac structures on a Jacobi bialgebroid $((A,\phi_0),(A^{*},X_0))$, we introduce the following pairings $(\cdot,\cdot )_{\pm}$ and bracket $[\![\cdot,\cdot]\!]$ on the Whitney sum $A\oplus A^{*}$:
 \begin{align*}
(X+\xi,Y+\eta )_{\pm} &:=\frac{1}{2}\left(\langle\xi,Y\rangle\pm \langle\eta,X\rangle \right);\\
[\![X+\xi,Y+\eta]\!]&:=([X,Y]_{A,\phi_0}+\mathcal{L}_{\xi}^{A^{*}\!,X_0}Y-\mathcal{L}_{\eta}^{A^{*}\!,X_0}X\\
                    &\qquad \qquad \qquad \qquad-d_{A^*\!,X_0}(X+\xi,Y+\eta)_-)\\
                    &\qquad +([\xi,\eta]_{A^{*},\!X_0}+\mathcal{L}_{X}^{A,\phi_0}\eta-\mathcal{L}_{Y}^{A,\phi_0}\xi\\
                    &\qquad \qquad \qquad \qquad+d_{A,\phi_0}(X+\xi,Y+\eta)_-);\\
\end{align*}

We notice that the pairings $(\cdot,\cdot )_{\pm}'$ and the bracket $[\![\cdot,\cdot]\!]'$ defined as above on $A^{*}\oplus A$ for a Jacobi bialgebroid $((A^{*},X_0),(A,\phi_0))$ satisfy
 \begin{align}\label{kakko-tachi}
(\cdot,\cdot )_{\pm}'=\pm(\cdot,\cdot )_{\pm},\quad [\![\cdot,\cdot]\!]'=[\![\cdot,\cdot]\!].
\end{align}

\begin{defn}[\cite{NC}]
Let $((A,\phi_0),(A^{*},X_0))$ be a Jacobi bialgebroid over $M$. A subbundle $L$ of $A\oplus A^{*}$ is a {\it Dirac structure} on $((A,\phi_0),(A^{*},X_0))$ if it is maximally isotropic under the pairing $(\cdot ,\cdot )_+$ and $\Gamma (L)$ is closed under the bracket $[\![\cdot,\cdot]\!]$. By (\ref{kakko-tachi}), Dirac structures on $((A,\phi_0),(A^{*},X_0))$ and on $((A^{*},X_0),(A,\phi_0))$ coincide. For a Jacobi algebroid $(A,\phi_0)$, we call a Dirac structure on a Jacobi bialgebroid $((A,\phi_0),(A^{*}_{0},0))$ in Example \ref{trivial example} a {\it Dirac structure on $(A,\phi_{0})$}. 
\end{defn}

 Let $((A,\phi_0),(A^{*},X_0))$ be a Jacobi bialgebroid over $M$, $\pi$ and $\omega $ elements in $\Gamma (\Lambda^2A)$ and in $\Gamma (\Lambda^2A^{*})$, respectively. We set
 \begin{align*}
\graph \pi^\sharp&:=\{\pi^\sharp\xi+\xi\ |\ \xi\in A^{*}\}\subset A\oplus A^{*},\\
\graph \omega^\flat&:=\{X+\omega X\ |\ X\in A\}\subset A\oplus A^{*}.
\end{align*}

\begin{theorem}[\cite{NC}]
 With the above notations, the graph $\graph \pi ^\sharp$ (resp. $\graph \omega ^\flat$) of a bundle map $\pi^\sharp$ (resp. $\omega^\flat$) is a Dirac structure on $((A,\phi_0),$ $(A^{*},X_0))$ if and only if $\pi $ (resp. $\omega$) satisfies the {\it Maurer-Cartan type equation}: 
\begin{align}\label{MC equation}
d_{A^{*}\!,X_0}\pi +\frac{1}{2}[\pi ,\pi]_{A,\phi_0}=0\quad \left(resp. \ d_{A,\phi_0}\omega +\frac{1}{2}[\omega ,\omega]_{A^{*}\!,X_0}=0\right).
\end{align}
 \end{theorem}
 
\begin{rem}
A Dirac structure on a Jacobi bialgebroid $((A,0),(A^{*},0))$ in Example \ref{Lie bialgebroids} is called a {\it Dirac structure on a Lie bialgebroid $(A,A^{*})$} and a Dirac structure on a Jacobi algebroid $(A,0)$ is called a {\it Dirac structure on a Lie algebroid $A$}. Then (\ref{MC equation}) coincides with the Maurer-Cartan type equation for a Lie bialgebroid $(A,A^{*})$ introduced in \cite{LWX}. 
\end{rem}

  \begin{example}\label{standard case M-C}
 For any Jacobi algebroid $(A,\phi_{0})$, the Maurer-Cartan type equations for $(A,\phi_{0})$ are
\begin{align*}
[\pi ,\pi]_{A,\phi_0}=0,\quad d_{A,\phi_0}\omega=0.
\end{align*} 
The former means that $\pi $ in $\Gamma (\Lambda^{2}A)$ is a Jacobi structure on $(A,\phi_{0})$ and the latter means that $\omega $ in $\Gamma (\Lambda^{2}A^{*})$ is a $\phi_{0}$-presymplectic structure on $(A,\phi_{0})$.
 \end{example}

 \subsection{Relations}\label{Relations}
 
 Let $U, V$ and $W$ be sets. We call subsets $R$ and $R'$ of  $U\times V$ and $V\times W$ {\it relations}. Then the {\it decomposition} $R'\ast R$ of $R$ and $R'$, and the {\it inverse} $\overline{R}$ of $R$ are defined by
 \begin{align*}
R'\ast R&:=\{(u,w)\in U\times W\,|\, \exists v\in V, (u,v)\in R\  \mbox{and}\ (v,w)\in R'\},\\
\overline{R}&:=\{(v,u)\in V\times U\,|\, (u,v)\in R\}.
 \end{align*}
 Moreover we set
 \begin{align*}
R'\Diamond R:=\{(u,v,w)\in U\times V\times W\,|\, (u,v)\in R, (v,w)\in R'\}.
 \end{align*}
 We notice that $\overline{R'\ast R}=\overline{R}\,\ast \,\overline{R'}$. Let $\phi :U\rightarrow V$ and $\phi ':V\rightarrow W$ be maps. Then we obtain $\graph \phi'\ast \graph \phi=\graph (\phi'\circ \phi)$. It is clear that $\overline{\graph \phi}=\graph (\phi^{-1})$ if $\phi$ is invertible.
 
 Let $U$ and $V$ be vector spaces. Then the {\it dual} $R^*$ of a relation $R\subset U\times V$ is defined by
 \begin{align*}
R^*:=\{(\beta,\alpha)\in V^*\times U^*\,|\, \forall(u,v)\in R,\langle \alpha ,u\rangle=\langle \beta ,v\rangle \}.
 \end{align*}
 We notice that $\overline{R^*}=\overline{R}^\ast$. Let $\phi :U\rightarrow V$ be a linear map. Then it is clear that $(\graph \phi)^*=\graph (\phi^*)$, where $\phi^*:V^*\rightarrow U^*$ is the dual map of $\phi$.
 
 Let $U$ and $V$ be vector bundles over a manifold $M$ and $R\subset U\times V$ a relation. Then the relation $\underline{R}\subset \Gamma (U)\times \Gamma (V)$ induced by $R$ is defined by
 \begin{align*}
\underline{R}:=\{(X,Y)\in \Gamma (U)\times \Gamma (V)\,|\, \forall p\in M, (X_p,Y_p)\in R\}.
 \end{align*}
 We notice that $\underline{R}=\Gamma(R)$ if $R\subset U\times V$ is a vector bundle over $M$. It is clear that $(\underline{R})^*=\underline{R^*}$. For any bundle map $\phi:U\rightarrow V$, we get $\underline{\graph \phi}=\graph {\underline{\phi}}$. Here $\underline{\phi}:\Gamma (U)\rightarrow \Gamma (V)$ is the map induced by $\phi$. In this paper, we often omit underline and denote the induced relation and map by the same symbols.

We define the Nijenhuis torsion of relations on a Lie algebroid.
 
 \begin{defn}[\cite{K}]\label{torsion and Nijenhuis}
Let $(A,[\cdot,\cdot]_A,\rho_A)$ be a Lie algebroid over $M$ and $R\subset A\times A$ a relation. Then the {\it Nijenhuis torsion} of $R$ is a function $\mathcal{T}_{R}:\underline{R}\times \underline{R}\times (\underline{R}^*\Diamond \underline{R}^*)\rightarrow \mathbb{R}$ defined by
 \begin{align*}
\mathcal{T}_{R}&((X_1,Y_1),(X_2,Y_2),(\alpha,\beta,\gamma))\\
             &:=\langle \alpha,[Y_1,Y_2]_A\rangle -\langle \beta,[Y_1,X_2]_A+[X_1,Y_2]_A\rangle+\langle \gamma,[X_1,X_2]_A\rangle
 \end{align*}
for all $(X_1,Y_1),(X_2,Y_2)$ in $\underline{R}$ and $(\alpha,\beta,\gamma)$ in $(\underline{R}^*\Diamond \underline{R}^*)$. A relation $R$ is {\it Nijenhuis} if $\mathcal{T}_{R}$ vanishes. 
 \end{defn}
\noindent 
 It follows easily that $\mathcal{T}_{R}=\mathcal{T}_{\overline{R}}$. A bundle map $N:A\rightarrow A$ over $M$ is a {\it Nijenhuis structure} on $A$ if the induced $(1,1)$-tensor $N$ in $\Gamma (A^{*}\otimes A)$ satisfies that
 \begin{align*}
\mathcal{T}_{N}(X,Y):=[NX,NY]_{A}-N[NX,Y]_{A}-N[X,NY]_{A}+N^2[X,Y]_{A}
 \end{align*}
vanishes for any $X$ and $Y$ in $\Gamma (A)$. A bundle map $N:A\rightarrow A$ over $M$ is a Nijenhuis structure on $A$ if and only if $\graph N$ is a Nijenhuis relation on $A$, i.e., $\mathcal{T}_{N}=0$ is equivalent with $\mathcal{T}_{\graph N}=0$. Hence Nijenhuis relations are a generalization of Nijenhuis structures.

 \subsection{Dirac pairs on Lie bialgebroids}\label{Dirac pairs on Lie bialgebroids}

For any relations $L$ and $L'\subset A\times A^{*}$, where $A$ and $A^{*}$ are vector bundles in duality over $M$, we set
 \begin{align}
\mathrm{N}_{L,L'}:=\overline{L}*L'.
\end{align}
Then we get $\mathrm{N}_{L,L'}=\overline{\mathrm{N}_{L',L}}$.
 
 \begin{defn}[\cite{K}]
 Let $(A,A^{*})$ be a Lie bialgebroid over $M$, $L$ and $L'$ Dirac structures on $(A,A^{*})$. Then $(L,L')$ is a {\it Dirac pair on} $(A, A^*)$ if $N_{L,L'}$ is a Nijenhuis relation. If $L$ and $L'$ are Dirac structures on $A$, then a pair $(L,L')$ is called {\it Dirac pair on} $A$. 
 \end{defn}
 
 \noindent
 Since $\mathcal{T}_{\mathrm{N}_{L,L'}}=\mathcal{T}_{\overline{\mathrm{N}_{L'\!,L}}}=\mathcal{T}_{\mathrm{N}_{L'\!,L}}$, if $(L,L')$ is a Dirac pair, then so is $(L',L)$.
 
 For any $2$-sections $\pi$ and $\pi'$ in $\Gamma (\Lambda^{2}A)$, the following holds.

 \begin{lemma}[\cite{K}]\label{2-section Nijenhuis torsion}
For any $\pi,\pi'$ in $\Gamma (\Lambda^2A)$, the Nijenhuis torsion of $\mathrm{N}_{L,L'}$, where $L:=\overline{\graph \pi^\sharp}$ and $L:=\overline{\graph \pi^{\prime\sharp}}$, satisfies the following:
 \begin{align*}
\mathcal{T}_{\mathrm{N}_{L,L'}}&((\pi^{\prime\sharp}\xi_1,\pi^\sharp\xi_1),(\pi^{\prime\sharp}\xi_2,\pi^\sharp\xi_2),(\xi,\xi',\xi''))\\
&=[\pi,\pi]_{A}(\xi_1,\xi_2,\xi)+[\pi',\pi']_{A}(\xi_1,\xi_2,\xi'')-2[\pi,\pi']_{A}(\xi_1,\xi_2,\xi').
\end{align*}
 \end{lemma}

Let $A$ be a Lie algebroid over $M$. Then a $2$-section on $A$ is Poisson if and only if its graph is a Dirac structure on $A$, i.e., $[\pi,\pi]_{A}=0$. A pair $(\pi,\pi')$ of two Poisson structures on $A$ is a {\it Poisson pair} if a pair $(\overline{\graph \pi^\sharp},\overline{\graph \pi^{\prime\sharp}})$ is a Dirac pair on $A$. A Poisson pair $(\pi,\pi^{\prime})$ is {\it non-degenerate}  if both $\pi $ and $\pi'$ are non-degenerate. Two Poisson structures $\pi$ and $\pi'$ on $A$ are compatible, i.e., $\pi+\pi'$ is also Poisson if and only if $[\pi,\pi']_A=0$. Therefore $(\pi,\pi')$ is a Poisson pair by Lemma \ref{2-section Nijenhuis torsion}. We call a pair $(\pi,\pi')$ cosisting of compatible Poisson structures a {\it Hamiltonian pair}. Conversely, if a Poisson pair $(\pi,\pi')$ satisfies
\begin{align}\label{compatible Poisson condition}
A^{*}=(\pi^{\sharp})^{-1}(\mathrm{Im}\,\pi^{\prime\sharp})\cap(\pi^{\prime\sharp})^{-1}(\mathrm{Im}\,\pi^{\sharp}),
 \end{align}
then $(\pi,\pi')$ is a Hamiltonian pair. In particular, since a non-degenerate Poisson pair $(\pi,\pi')$ satisfies (\ref{compatible Poisson condition}), $(\pi,\pi')$ is a Hamiltonian pair. It is well known that Hamiltonian pairs themselves are induced by several compatible structures, for example, Poisson-Nijenhuis structures \cite{KMa}, \cite{V}, pairs of Poisson structures and their complementary $2$-forms \cite{V2} and so on.

A $2$-cosection on $A$ is presymplectic, i.e., it is $d_A$-closed, if and only if its graph is a Dirac structure on $A$. A pair $(\omega,\omega')$ of two presymplectic structures on $A$ is a {\it presymplectic pair} if a pair $(\graph \omega^\flat,\graph \omega^{\prime\flat})$ is a Dirac pair. A presymplectic pair $(\omega,\omega')$ is {\it symplectic pair} if both $\omega$ and $\omega'$ are symplectic. The following proposition for Poisson and presymplectic pairs hold.


 \begin{prop}[\cite{K}]\label{presymplectic pair property}
Symplectic pairs are in one-to-one correspondence with non-degenerate Poisson pairs.
 \end{prop}
 
 We define P$\Omega$- and $\Omega$N-structures on a Lie algebroid.
 
\begin{defn}[\cite{K}]
 Let $A$ be a Lie algebroid over $M$, $\pi$ a $2$-section on $A$ and $\omega$ a $2$-cosection on $A$. Then a pair $(\pi,\omega)$ is a {\it P$\Omega$-structure} on $A$ if $\pi$ is Poisson and both $\omega$ and $\omega_{N}$ are $d_{A}$-closed, where $N:=\pi^{\sharp}\circ \omega^{\flat}$ and $\omega _{N}$ is a $2$-cosection characterized by $\omega_{N}^{\flat}=\omega^{\flat}\circ N$.  
\end{defn}

\begin{defn}[\cite{K}]
 Let $A$ be a Lie algebroid over $M$, $\omega$ a $2$-cosection on $A$ and $N$ a $(1,1)$-tensor field on $A$. Then a pair $(\omega,N)$ is an {\it $\Omega$N-structure} on $A$ if $\omega^{\flat}\circ N=N^{*}\circ \omega^{\flat}$, $N$ is Nijenhuis, i.e., $\mathcal{T}_{N}=0$, and both $\omega$ and $\omega_{N}$ are $d_{A}$-closed, where $\omega _{N}$ is a $2$-cosection characterized by $\omega_{N}^{\flat}=\omega^{\flat}\circ N$. We can also define a {\it weak $\Omega$N-structure} on $A$ by replacing $\mathcal{T}_{N}=0$ with $\omega^{\flat}(\mathcal{T}_{N}(X,Y))=0$ for any $X$ and $Y$ in $\Gamma (A)$.
\end{defn}

\noindent
For a weak $\Omega$N-structure $(\omega, N)$ on $A$, we set
\begin{align*}
\mathrm{N}_{(\omega,N)}^{+}:=\{(\omega^{\flat}X,\omega^{\flat}_{N}X)\,|\, X\in A\}\subset \mathrm{N}_{L,L'}^{*},
\end{align*}
where $L:=\graph\omega^{\flat}$ and $L':=\graph\omega_{N}^{\flat}$.

These structures are characterized in terms of Dirac pairs on Lie algebroids.

 \begin{prop}[\cite{K}]\label{POmega to Dirac}
 Let $A$ be a Lie algebroid over $M$, $\pi$ a Poisson structure on $A$ and $\omega$ a presymplectic structure on $A$.
\begin{enumerate}
\item[{\rm (i)}] If a pair $(\pi,\omega )$ is a P$\Omega$-structure on $A$, then a pair $(\overline{\graph \pi^{\sharp}},\graph\omega^{\flat})$ is a Dirac pair on $A$;
\item[{\rm (ii)}] Conversely, if $(\overline{\graph \pi^{\sharp}},\graph\omega^{\flat})$ is a Dirac pair on $A$, and if $\pi $ is nondegenerate, then a pair $(\pi,\omega )$ is a P$\Omega$-structure on $A$.
\end{enumerate}
 \end{prop}

 \begin{prop}[\cite{K}]\label{weak OmegaN to Dirac}
 Let $A$ be a Lie algebroid over $M$, $\omega$ a presymplectic structure on $A$ and $N$ a $(1,1)$-tensor field on $A$.
\begin{enumerate}
\item[{\rm (i)}] If a pair $(\omega ,N)$ is an $\Omega$N-structure on $A$, and if $\mathrm{N}_{(\omega,N)}^{+}=\mathrm{N}_{L,L'}^{*}$, where $L:=\graph\omega^{\flat}$ and $L':=\graph\omega_{N}^{\flat}$, then a pair $(L,L')$ is a Dirac pair on $A$;
\item[{\rm (ii)}] Conversely, if $(\graph \omega^{\flat},\graph\omega_{N}^{\flat})$ is a Dirac pair on $A$, then a pair $(\omega ,N)$ is a weak $\Omega$N-structure on $A$.
\end{enumerate}
 \end{prop}

\section{Dirac pairs on Jacobi bialgebroids}\label{Dirac pairs on Jacobi bialgebroids}

In this section, 
we generalize Dirac pairs on a Lie bialgebroid and introduce Dirac pairs on a Jacobi bialgebroid. We prove that similar properties for Dirac pairs on a Lie bialgebroid also hold for them on a Jacobi bialgebroid.
 
 We start with the definition of Dirac pairs on Jacobi bialgebroids. This is defined as is the case on Lie bialgebroids.
 
 \begin{defn}
 Let $((A,\phi_{0}),(A^{*},X_{0}))$ be a Jacobi bialgebroid over $M$, $L$ and $L'$ Dirac structures on $((A,\phi_{0}),(A^{*},X_{0}))$. Then $(L,L')$ is a {\it Dirac pair on} $((A,\phi_{0}),(A^{*},X_{0}))$ if $\mathrm{N}_{L,L'}$ is a Nijenhuis relation. If $L$ and $L'$ are Dirac structures on $(A,\phi_{0})$, then a pair $(L,L')$ is called {\it Dirac pair on} $(A,\phi_{0})$.
 \end{defn}
 
 \noindent
  Since $\mathcal{T}_{\mathrm{N}_{L,L'}}=\mathcal{T}_{\overline{\mathrm{N}_{L',L}}}=\mathcal{T}_{\mathrm{N}_{L',L}}$, if $(L,L')$ is a Dirac pair, then so is $(L',L)$.
 
 As on Lie bialgebroids, we obtain the following property of Nijenhuis relations on Jacobi bialgebroids.
 
  \begin{lemma}\label{pi and pi' Nijenhuis torsion}
Let $(A,\phi_{0})$ be a Jacobi algebroid. For any $\pi,\pi'$ in $\Gamma (\Lambda^2A)$, the Nijenhuis torsion of $\mathrm{N}_{L,L'}$, where $L:=\overline{\graph \pi^\sharp}$ and $L:=\overline{\graph \pi^{\prime\sharp}}$, satisfies the following:
\begin{align*}
\mathcal{T}_{\mathrm{N}_{L,L'}}&((\pi^{\prime\sharp}\xi_1,\pi^\sharp\xi_1),(\pi^{\prime\sharp}\xi_2,\pi^\sharp\xi_2),(\xi,\xi'\xi''))\\
&=[\pi,\pi]_{A,\phi_0}(\xi_1,\xi_2,\xi)+[\pi',\pi']_{A,\phi_0}(\xi_1,\xi_2,\xi'')-2[\pi,\pi']_{A,\phi_0}(\xi_1,\xi_2,\xi').
\end{align*}
 \end{lemma}
\noindent
By using (\ref{Jacobi structure bracket property}), Lemma \ref{pi and pi' Nijenhuis torsion} can be shown exactly as Lemma \ref{2-section Nijenhuis torsion}.
 
 In general, Dirac pairs on a Jacobi bialgebroid and on a Lie bialgebroid have the following relationship. This is one of the main theorems in this paper.

 \begin{theorem}\label{main 1}
Let $((A,\phi_{0}),(A^{*},X_{0}))$ be a Jacobi bialgebroid over $M$, $\pi_{i}$ a 2-section on $A$ satisfying the Maurer-Cartan type equation and $\omega_{i}$ a 2-cosection on $A$ satisfying the Maurer-Cartan type equation ($i=1,2$). Let $(\tilde{A}_{\bar{\phi}_{0}},\tilde{A}^{*}_{\hat{X}_{0}})$ be the induced Lie bialgebroid over $M\times \mathbb{R}$ (Proposition \ref{Jacobi and Lie bilagebroid}). We set $\tilde{\pi}_{i}:=e^{-t}\pi_{i}$ in $\Gamma (\Lambda ^{2}\tilde{A})$ and $\tilde{\omega}_{i}:=e^{t}\omega_{i}$ in $\Gamma (\Lambda ^{2}\tilde{A}^{*})$, where $t$ is the standard coordinate in $\mathbb{R}$. Then:
  \begin{enumerate}
 \item[(i)] $(\overline{\graph\tilde{\pi}_{1}^{\sharp}},\overline{\graph\tilde{\pi}_{2}^{\sharp}})$ is a Dirac pair on $\left(\tilde{A}_{\bar{\phi}_{0}},\tilde{A}^{*}_{\hat{X}_{0}}\right)$ if and only if $(\overline{\graph\pi_{1}^{\sharp}},$ $\overline{\graph\pi_{2}^{\sharp}})$ is a Dirac pair on $((A,\phi_{0}),(A^{*},X_{0}))$;
 \item[(ii)] $(\overline{\graph\tilde{\pi}_{1}^{\sharp}},\graph\tilde{\omega}_{2}^{\flat})$ is a Dirac pair on $\left(\tilde{A}_{\bar{\phi}_{0}},\tilde{A}^{*}_{\hat{X}_{0}}\right)$ if and only if $(\overline{\graph\pi_{1}^{\sharp}},$ $\graph\omega_{2}^{\flat})$ is a Dirac pair on $((A,\phi_{0}),(A^{*},X_{0}))$;
 \item[(iii)] $(\graph\tilde{\omega}_{1}^{\flat},\graph\tilde{\omega}_{2}^{\flat})$ is a Dirac pair on $\left(\tilde{A}_{\bar{\phi}_{0}},\tilde{A}^{*}_{\hat{X}_{0}}\right)$ if and only if $(\graph\omega_{1}^{\flat},$ $\graph\omega_{2}^{\flat})$ is a Dirac pair on $((A,\phi_{0}),(A^{*},X_{0}))$.
  \end{enumerate}
 \end{theorem}

\noindent
In order to prove Theorem \ref{main 1}, we need the following lemmas.
 
\begin{lemma}\label{Jacobi and Poisson structures}
In the notation of Theorem \ref{main 1}, $\graph\tilde{\pi}_{1}^{\sharp}$ (resp. $\graph\tilde{\omega}_{1}^{\flat}$) is a Dirac structure on $\left(\tilde{A}_{\bar{\phi}_{0}},\tilde{A}^{*}_{\hat{X}_{0}}\right)$ if and only if $\graph\pi_{1}^{\sharp}$ (resp. $\graph\omega_{1}^{\flat}$) is a Dirac structure on $((A,\phi_{0}),(A^{*},X_{0}))$.
\end{lemma}

\begin{proof}
Let $[\cdot,\cdot\bar{]}_{\tilde{\pi}_{1}}^{\phi_{0}}$ $\left(\mbox{resp.}\ [\cdot,\cdot\hat{]}^{X_0}_{\tilde{\omega}_{1}}\right)$ be the bracket on $\Gamma (\tilde{A}^{*})$ (resp. $\Gamma (\tilde{A})$) induced by a 2-section $\tilde{\pi}_{1}$ on $\tilde{A}$ (resp. a 2-cosection $\tilde{\omega}_{1}$ on $\tilde{A}^{*}$).
 By long calculations,
we see that
\begin{align}\label{kakko tilde umu kankei}
[\tilde{\pi}_{1},\tilde{\pi}_{1}\bar{]}_A^{\phi_0}&=e^{-2t}[\pi_{1},\pi_{1}]_{A,\phi_0},\\
\hat{d}_{A^{*}}^{X_{0}}\tilde{\pi}_{1}&=e^{-2t}d_{A^{*}\!,X_{0}}\pi_{1}
\end{align}
hold on $\Gamma (\tilde{A}^{*})$. Therefore we obtain 
\begin{align}
\hat{d}_{A^{*}}^{X_{0}}\tilde{\pi}_{1}+\frac{1}{2}[\tilde{\pi}_{1},\tilde{\pi}_{1}\bar{]}_A^{\phi_0}&=e^{-2t}\left(d_{A^{*}\!,X_{0}}\pi_{1}+\frac{1}{2}[\pi_{1},\pi_{1}]_{A,\phi_0}\right)
\end{align}
on $\Gamma (\tilde{A}^{*})$. Since $\Gamma (\tilde{A}^{*})$ can be regarded as the set of curves in $\Gamma (A^{*})$, $\tilde{\pi}_{1}$ satisfies the Maurer-Cartan type equation for $\left(\tilde{A}_{\bar{\phi}_{0}},\tilde{A}^{*}_{\hat{X}_{0}}\right)$ if and only if $\pi_1$ satisfies the Maurer-Cartan type equation for $((A,\phi_{0}),(A^{*},X_{0}))$. 

Similarly, by long calculations, we see that 
\begin{align*}
[\tilde{\omega}_{1},\tilde{\omega}_{1}\hat{]}_{A^{*}}^{X_0}&=e^{t}[\omega_{1},\omega_{1}]_{A^{*}\!,X_0},\\
\bar{d}_{A}^{\phi_{0}}\tilde{\omega}_{1}&=e^{t}d_{A,\phi_{0}}\omega_{1}
\end{align*}
holds on $\Gamma (\tilde{A})$. Therefore we obtain 
\begin{align*}
\bar{d}_{A}^{\phi_{0}}\tilde{\omega}_{1}+\frac{1}{2}[\tilde{\omega}_{1},\tilde{\omega}_{1}\hat{]}_{A^{*}}^{X_0}=e^{t}\left(d_{A,\phi_{0}}\omega_{1}+\frac{1}{2}[\omega_{1},\omega_{1}]_{A^{*}\!,X_0}\right)
\end{align*}
on $\Gamma (\tilde{A})$. Since $\Gamma (\tilde{A})$ can be regarded as the set of curves in $\Gamma (A)$, $\tilde{\omega}_{1}$ satisfies the Maurer-Cartan type equation for $\left(\tilde{A}_{\bar{\phi}_{0}},\tilde{A}^{*}_{\hat{X}_{0}}\right)$ if and only if $\omega_1$ satisfies the Maurer-Cartan type equation for $((A,\phi_{0}),(A^{*},X_{0}))$.
\end{proof}
 
\begin{lemma}\label{Nijenhuis structure property}
In the notation of Theorem \ref{main 1}, any $(1,1)$-tensor field $N$ on $\tilde{A}_{\bar{\phi}_{0}}$ independent of $t$ can be regarded as a $(1,1)$-tensor field on $A$ in general. Then $N$ is Nijenhuis on $\tilde{A}_{\bar{\phi}_{0}}$ if and only if $N$ is Nijenhuis on $A$.
\end{lemma} 

\begin{proof}
We set the Nijenhuis torsion of $N$ on $\tilde{A}_{\bar{\phi}_{0}}$ and $A$ by $\mathcal{T}_{N}^{\tilde{A}_{\bar{\phi}_{0}}}$ and $\mathcal{T}_{N}^{A}$ respectively. By a straightforward calculation, we have for any $\tilde{X}$ and $\tilde{Y}$ in $\Gamma (\tilde{A})$, 
\begin{align*}
\mathcal{T}_{N}^{\tilde{A}_{\bar{\phi}_{0}}}(\tilde{X},\tilde{Y})=\mathcal{T}_{N}^{A}(\tilde{X},\tilde{Y}).
\end{align*}
Since $\tilde{X}$ and $\tilde{Y}$ in $\Gamma (\tilde{A})$ can be regarded as curves in $\Gamma (A)$, $\mathcal{T}_{N}^{\tilde{A}_{\bar{\phi}_{0}}}=0$ is equivalent with $\mathcal{T}_{N}^{A}=0$.
\end{proof}

\begin{proof}[Proof of Theorem \ref{main 1}]
We prove (i). We set
\begin{align*}
L_{i}:=\overline{\graph \pi_{i}^{\sharp}},\quad \tilde{L}_{i}:=\overline{\graph \tilde{\pi}_{i}^{\sharp}}\quad (i=1,2).
\end{align*} 
By Lemma \ref{2-section Nijenhuis torsion}, Lemma \ref{pi and pi' Nijenhuis torsion} and the equation (\ref{kakko tilde umu kankei}), we compute
\begin{align*}
\mathcal{T}_{\mathrm{N}_{\tilde{L}_{1},\tilde{L}_{2}}}&((\tilde{\pi}_{2}^{\sharp}\tilde{\xi}_1,\tilde{\pi}_{1}^{\sharp}\tilde{\xi}_{1}),(\tilde{\pi}_{2}^{\sharp}\tilde{\xi}_2,\tilde{\pi}_{1}^{\sharp}\tilde{\xi}_2),(\tilde{\xi},\tilde{\xi}',\tilde{\xi}''))\\
                                      &=[\tilde{\pi}_{1},\tilde{\pi}_{1}\bar{]}_A^{\phi_0}(\tilde{\xi}_1,\tilde{\xi}_2,\tilde{\xi})+[\tilde{\pi}_{2},\tilde{\pi}_{2}\bar{]}_A^{\phi_0}(\tilde{\xi}_1,\tilde{\xi}_2,\tilde{\xi}'')-2[\tilde{\pi}_{1},\tilde{\pi}_{2}\bar{]}_A^{\phi_0}(\tilde{\xi}_1,\tilde{\xi}_2,\tilde{\xi}')\\
                                      &=e^{-2t}[\pi_{1},\pi_{1}]_{A,\phi_0}(\tilde{\xi}_1,\tilde{\xi}_2,\tilde{\xi})+e^{-2t}[\pi_{2},\pi_{2}]_{A,\phi_0}(\tilde{\xi}_1,\tilde{\xi}_2,\tilde{\xi}'')\\
                                      &\qquad \qquad \qquad \qquad \qquad \qquad \qquad \qquad -2e^{-2t}[\pi_{1},\pi_{2}]_{A,\phi_0}(\tilde{\xi}_1,\tilde{\xi}_2,\tilde{\xi}')\\
                                      &=e^{-2t}\mathcal{T}_{\mathrm{N}_{L_{1},L_{2}}}((\pi_{2}^{\sharp}\tilde{\xi}_1,\pi_{1}^{\sharp}\tilde{\xi}_{1}),(\pi_{2}^{\sharp}\tilde{\xi}_2,\pi_{1}^{\sharp}\tilde{\xi}_2),(\tilde{\xi},\tilde{\xi}',\tilde{\xi}'')).
\end{align*}
Since any element in $\tilde{L}_{i}$ can be regarded as a curve in $L_{i}$, it is clear that any element in $\mathrm{N}_{\tilde{L}_{1},\tilde{L}_{2}}$ or $\mathrm{N}_{\tilde{L}_{1},\tilde{L}_{2}}^{*}\Diamond \mathrm{N}_{\tilde{L}_{1},\tilde{L}_{2}}^{*}$ can also be regarded as a curve in $\mathrm{N}_{L_{1},L_{2}}$ or $\mathrm{N}_{L_{1},L_{2}}^{*}\Diamond \mathrm{N}_{L_{1},L_{2}}^{*}$, respectively. Therefore the condition $\mathcal{T}_{\mathrm{N}_{\tilde{L}_{1},\tilde{L}_{2}}}=0$ is equivalent with the condition $\mathcal{T}_{\mathrm{N}_{L_{1},L_{2}}}=0$. This means (i).

For any $\tilde{X}$ in $\Gamma (\tilde{A})$, we have
\begin{align*}
(\tilde{\pi}_{1}^{\sharp}\circ \tilde{\omega}_{2}^{\flat})(\tilde{X})&=\tilde{\pi}_{1}^{\sharp}(\tilde{\omega}_{2}^{\flat}\tilde{X})=e^{-t}\pi_{1}^{\sharp}(e^{t}\omega_{2}^{\flat}\tilde{X})={\pi}_{1}^{\sharp}(\omega_{2}^{\flat}\tilde{X})=(\pi_{1}^{\sharp}\circ \omega_{2}^{\flat})(\tilde{X}).
\end{align*}
Hence the $(1,1)$-tensor field $\tilde{\pi}_{1}^{\sharp}\circ \tilde{\omega}_{2}^{\flat}=\pi_{1}^{\sharp}\circ \omega_{2}^{\flat}$ is independent of $t$. By the definition, $(\overline{\graph\tilde{\pi}_{1}^{\sharp}},\graph\tilde{\omega}_{2}^{\flat})$ is a Dirac pair on $\left(\tilde{A}_{\bar{\phi}_{0}},\tilde{A}^{*}_{\hat{X}_{0}}\right)$ if and only if 
\begin{align*}
\mathrm{N}_{\tilde{L}_{1},\tilde{L}_{2}'}=\graph\tilde{\pi}_{1}^{\sharp}*\graph\tilde{\omega}_{2}^{\flat}=\graph(\tilde{\pi}_{1}^{\sharp}\circ\tilde{\omega}_{2}^{\flat})=\graph(\pi_{1}^{\sharp}\circ\omega_{2}^{\flat})
\end{align*}
is a Nijenhuis relation on $\tilde{A}_{\bar{\phi}_{0}}$, where we set $\tilde{L}_{1}:=\overline{\graph\tilde{\pi}_{1}^{\sharp}}$ and $\tilde{L}_{2}':=\graph\tilde{\omega}_{2}^{\flat}$. This condition is equivalent with the condition $\pi_{1}^{\sharp}\circ\omega_{2}^{\flat}$ is a Nijenhuis structure on $\tilde{A}_{\bar{\phi}_{0}}$. Then by Lemma \ref{Nijenhuis structure property}, this is equivalent with that $\pi_{1}^{\sharp}\circ\omega_{2}^{\flat}$ is a Nijenhuis structure on $A$. Similarly, $(\overline{\graph\pi_{1}^{\sharp}},\graph\omega_{2}^{\flat})$ is a Dirac pair on $((A,\phi_{0}),(A^{*},X_{0}))$ if and only if $\pi_{1}^{\sharp}\circ\omega_{2}^{\flat}$ is a Nijenhuis structure on $A$. Therefore we obtain (ii).

Finally we prove (iii). We set $L_{i}':=\graph\omega_{i}^{\flat}$ and $\tilde{L}_{i}':=\graph\tilde{\omega}_{i}^{\flat}$. A pair $(\tilde{X},\tilde{Y})$ belongs to $\mathrm{N}_{\tilde{L}_{1}',\tilde{L}_{2}'}$ if and only if $\tilde{\omega}_{1}^{\flat}\tilde{X}=\tilde{\omega}_{2}^{\flat}\tilde{Y}$ holds by the definition. By differentiating both sides of $\tilde{\omega}_{1}^{\flat}\tilde{X}=\tilde{\omega}_{2}^{\flat}\tilde{Y}$ with respect to $t$, we obtain
\begin{align*}
\tilde{\omega}_{1}^{\flat}\tilde{X}+\tilde{\omega}_{1}^{\flat}\frac{\partial \tilde{X}}{\partial t}=\tilde{\omega}_{2}^{\flat}\tilde{Y}+\tilde{\omega}_{2}^{\flat}\frac{\partial \tilde{Y}}{\partial t}.
\end{align*}
Therefore $\tilde{\omega}_{1}^{\flat}\displaystyle\frac{\partial \tilde{X}}{\partial t}=\tilde{\omega}_{2}^{\flat}\displaystyle\frac{\partial \tilde{Y}}{\partial t}$ holds since $\tilde{\omega}_{i}=e^t\omega_i$ and $\tilde{\omega}_{1}^{\flat}\tilde{X}=\tilde{\omega}_{2}^{\flat}\tilde{Y}$. This means that a pair $\left(\displaystyle\frac{\partial \tilde{X}}{\partial t},\displaystyle\frac{\partial \tilde{Y}}{\partial t}\right)$ belongs to $\mathrm{N}_{\tilde{L}_{1}',\tilde{L}_{2}'}$. Therefore we obtain
\begin{align}\label{hishigata no gen}
\left\langle \tilde{\xi}_{1},\frac{\partial \tilde{Y}}{\partial t}\right\rangle=\left\langle \tilde{\xi}_{2},\frac{\partial \tilde{X}}{\partial t}\right\rangle, \quad \left\langle \tilde{\xi}_{2},\frac{\partial \tilde{Y}}{\partial t}\right\rangle=\left\langle \tilde{\xi}_{3},\frac{\partial \tilde{X}}{\partial t}\right\rangle
\end{align}
for any $(\tilde{X},\tilde{Y})$ in $\mathrm{N}_{\tilde{L}_{1}',\tilde{L}_{2}'}$ and $(\tilde{\xi}_{1},\tilde{\xi}_{2},\tilde{\xi}_{3})$ in $\mathrm{N}_{\tilde{L}_{1}',\tilde{L}_{2}'}^{*}\Diamond$ $\mathrm{N}_{\tilde{L}_{1}',\tilde{L}_{2}'}^{*}$. By using equations (\ref{hishigata no gen}), it follows that $\mathcal{T}_{\mathrm{N}_{\tilde{L}_{1}',\tilde{L}_{2}'}}=\mathcal{T}_{\mathrm{N}_{L_{1}',L_{2}'}}$ on $\mathrm{N}_{\tilde{L}_{1}',\tilde{L}_{2}'}\times\mathrm{N}_{\tilde{L}_{1}',\tilde{L}_{2}'}\times\mathrm{N}_{\tilde{L}_{1}',\tilde{L}_{2}'}^{*}\Diamond$ $\mathrm{N}_{\tilde{L}_{1}',\tilde{L}_{2}'}^{*}$.

Since $\mathrm{N}_{\tilde{L}_{1}',\tilde{L}_{2}'}$ and $\mathrm{N}_{\tilde{L}_{1}',\tilde{L}_{2}'}^{*}\Diamond$ $\mathrm{N}_{\tilde{L}_{1}',\tilde{L}_{2}'}^{*}$ are the sets of all curves in $\mathrm{N}_{L_{1}',L_{2}'}$ and $\mathrm{N}_{L_{1}',L_{2}'}^{*}\Diamond$ $\mathrm{N}_{L_{1}',L_{2}'}^{*}$ respectively, we see that $\mathcal{T}_{\mathrm{N}_{\tilde{L}_{1}',\tilde{L}_{2}'}}=0$ and $\mathcal{T}_{\mathrm{N}_{L_{1}',L_{2}'}}=0$ are equivalent.
\end{proof}

\begin{rem}
It follows immediately that $(\graph\tilde{\omega}_{1}^{\flat},\overline{\graph\tilde{\pi}_{2}^{\sharp}})$ is a Dirac pair on $\left(\tilde{A}_{\bar{\phi}_{0}},\tilde{A}^{*}_{\hat{X}_{0}}\right)$ if and only if $(\graph\omega_{1}^{\flat}$, $\overline{\graph\pi_{2}^{\sharp}})$ is a Dirac pair on $((A,\phi_{0}),(A^{*},X_{0}))$ by (iii) in Theorem \ref{main 1} and the fact that if a pair $(L,L')$ is a Dirac pair on $\left(\tilde{A}_{\bar{\phi}_{0}},\tilde{A}^{*}_{\hat{X}_{0}}\right)$ (resp. $((A,\phi_{0}),(A^{*},X_{0}))$), so is $(L',L)$.
\end{rem}
 
 \section{Dirac pairs on Jacobi algebroids}\label{Dirac pairs on Jacobi algebroids}

In this section, we consider a Dirac pair $(L,L')$ on a Jacobi algebroid $(A,\phi_{0})$, i.e., a pair consisting of two Dirac structures $L$ and $L'$ on $(A,\phi_{0})$ such that $\mathrm{N}_{L,L'}$ is a Nijenhuis relation on $(A,\phi_{0})$. 

 \subsection{Jacobi and $\phi_0$-presymplectic pairs on Jacobi algebroids}\label{Jacobi and phi_0-presymplectic pairs on Jacobi algebroids}
  
In this subsection, we investigate Jacobi and $\phi_{0}$-presymplectic pairs, which are defined by using Dirac pairs on $(A,\phi_{0})$. We show that these pairs have properties similar to Poisson and presymplectic pairs on a Lie algebroid, respectively. In addition, we show the relationship between Jacobi and Poisson pairs.

By Example \ref{standard case M-C}, for any Jacobi algebroid $(A,\phi_{0})$, a $2$-section $\pi $ in $\Gamma (\Lambda^{2}A)$ is a Jacobi structure on $(A,\phi_{0})$ if and only if $\graph \pi^{\sharp}$ is a Dirac structure on $(A,\phi_{0})$. Similarly, a $2$-cosection $\omega $ in $\Gamma (\Lambda^{2}A^{*})$ is a $\phi_{0}$-presymplectic structure on $(A,\phi_{0})$ if and only if $\graph \omega^{\flat}$ is a Dirac structure on $(A,\phi_{0})$. We define Jacobi and $\phi_{0}$-presymplectic pairs as analogy of Poisson and presymplectic pairs on a Lie algebroid.

\begin{defn}\label{Jacobi and phi0-presymp pairs}
Let $(A,\phi_{0})$ be a Jacobi algebroid, $\pi_{i}$ a Jacobi structure on $(A,\phi_{0})$ and $\omega_{i}$ a $\phi_{0}$-presymplectic structure on $(A,\phi_{0})$ $(i=1,2)$.
\begin{enumerate}
 \item[(i)] A pair $(\pi_{1},\pi_{2})$ is a {\it Jacobi pair} if a pair $(\overline{\graph\pi_{1}^{\sharp}},\overline{\graph\pi_{2}^{\sharp}})$ is a Dirac pair on $(A,\phi_{0})$. A Jacobi pair $(\pi_{1},\pi_{2})$ is {\it non-degenerate}  if both $\pi_{1} $ and $\pi_{2}$ are non-degenerate;
 \item[(ii)] A pair $(\omega_{1},\omega_{2})$ is a {\it $\phi_{0}$-presymplectic pair} if a pair $(\graph \omega_1^{\flat},\graph \omega_2^{\flat})$ is a Dirac pair on $(A,\phi_{0})$. A $\phi_{0}$-presymplectic pair consitinig of two $\phi_{0}$-symplectic pair is called a {\it $\phi_{0}$-symplectic pair}.
  \end{enumerate}
\end{defn}

It follows immediately from Lemma \ref{pi and pi' Nijenhuis torsion} and Definition \ref{Jacobi and phi0-presymp pairs} that a pair $(\pi_{1},\pi_{2})$ consisting of compatible Jacobi structures is a Jacobi pair. Conversely, if a Jacobi pair $(\pi_{1},\pi_{2})$ satisfies
\begin{align}\label{compatible Jacobi condition}
A^{*}=(\pi_{1}^{\sharp})^{-1}(\mathrm{Im}\,\pi_{2}^{\sharp})\cap(\pi_{2}^{\sharp})^{-1}(\mathrm{Im}\,\pi_{1}^{\sharp}),
 \end{align}
then $\pi_{1}$ and $\pi_{2}$ are compatible. In particular, since a non-degenerate Jacobi pair $(\pi_{1},\pi_{2})$ satisfies (\ref{compatible Jacobi condition}), two Jacobi structures $\pi_{1}$ and $\pi_{2}$ are compatible. It is well known that compatible two Jacobi structures themselves are induced by Jacobi-Nijenhuis structures \cite{MMP}, \cite{CNN} and so on.
 
 The following proposition is a relationship between $\phi_{0}$-symplectic pairs and non-degenerate Jacobi pairs. The proof is similar to Proposition \ref{presymplectic pair property} (\cite{K}).
  
 \begin{prop}
There exists a one-to-one correspondence between $\phi_{0}$-symplectic pairs and non-degenerate Jacobi pairs.
 \end{prop}
 
The following example is analogy of Example 3.5 in \cite{K}.

 \begin{example}\label{cotan r2 times r}
Let $M:=T^{*}\mathbb{R}^{2}\times \mathbb{R}$ and $\beta$ be the canonical contact form on $M$. In canonical coordinates $(x_{1},x_{2},y_{1},y_{2},z)$ on $M$, $\beta=-\sum_{i}^{}y_{i}dx_{i}+dz$. Explicit examples $(\Omega,\omega_{H}), (\Omega,\omega_{E})$ and $(\Omega,\omega_{P})$ of $(0,1)$-presymplectic pairs on $(TM\oplus \mathbb{R},(0,1))$ are defined by
\begin{align*}
\Omega&:=(d\beta,\beta)\\
\omega_{\mathrm{H}}&:=(d\beta_{\mathrm{H}},\beta_{\mathrm{H}}),\ \beta_{\mathrm{H}}:=-y_{1}dx_{1}+y_{2}dx_{2}+dz\\
\omega_{\mathrm{E}}&:=(d\beta_{\mathrm{E}},\beta_{\mathrm{E}}),\ \beta_{\mathrm{E}}:=-y_{2}dx_{1}+y_{1}dx_{2}+dz\\
\omega_{\mathrm{P}}&:=(d\beta_{\mathrm{P}},\beta_{\mathrm{P}}),\ \beta_{\mathrm{P}}:=-y_{2}dx_{1}+dz,
 \end{align*}
where the indices $\mathrm{H},\mathrm{E}$ and $\mathrm{P}$ mean ``hyperbolic'', ``elliptic'' and ``parabolic'', respectively. These mean that the restriction of a $2$-form $d\beta_{\mathrm{H}}$ (resp. $d\beta_{\mathrm{E}}$, $d\beta_{\mathrm{P}}$) to $T^{*}\mathbb{R}^{2}$ corresponds to a hyperbolic (resp. elliptic, parabolic) Monge-Amp$\grave{\mathrm{e}}$re equation with constant coefficients \cite{K} (see also \cite{LRC}, \cite{KLR}). By Example \ref{contact example}, $\Omega $ is a $(0,1)$-symplectic structure on $(TM\oplus\mathbb{R},(0,1))$ and satisfies $\Omega ^{3}=(0,(d\beta)^2\wedge \beta)\neq0$. Since $\omega_{\mathrm{H}}$ (resp. $\omega_{\mathrm{E}}$) satisfies $d_{TM\oplus\mathbb{R},(0,1)}\omega_{\mathrm{H}}=0$ (resp. $d_{TM\oplus\mathbb{R},(0,1)}\omega_{\mathrm{E}}=0$) and $\omega_{\mathrm{H}}^{3}=-\Omega^{3}$ (resp. $\omega_{\mathrm{E}}^{3}=\Omega^{3}$), $\omega_{\mathrm{H}}$ (resp. $\omega_{\mathrm{E}}$) is a $(0,1)$-symplectic structure on $(TM\oplus\mathbb{R},(0,1))$. On the other hand, since $\omega_{\mathrm{P}}$ satisfies $d_{TM\oplus\mathbb{R},(0,1)}\omega_{\mathrm{P}}=0$ and $\omega_{\mathrm{P}}^{3}=0$, $\omega_{\mathrm{P}}$ is not a $(0,1)$-symplectic structure but a $(0,1)$-presymplectic structure on $(TM\oplus\mathbb{R},(0,1))$. In order to show that pairs $(\Omega,\omega_{H}),(\Omega,\omega_{E})$ and $(\Omega,\omega_{P})$ are $(0,1)$-presymplectic, since
\begin{align*}
\mathrm{N}_{L,L_{\mathrm{H}}}&=\overline{\graph\Omega^{\flat}}*\graph\omega_{\mathrm{H}}^{\flat}=\graph(\Omega^{\flat})^{-1}*\graph\omega_{\mathrm{H}}^{\flat}\\
                             &=\graph((\Omega^{\flat})^{-1}\circ\omega_{\mathrm{H}}^{\flat}),\\
\mathrm{N}_{L,L_{\mathrm{E}}}&=\graph((\Omega^{\flat})^{-1}\circ\omega_{\mathrm{E}}^{\flat}),\\
\mathrm{N}_{L,L_{\mathrm{P}}}&=\graph((\Omega^{\flat})^{-1}\circ\omega_{\mathrm{P}}^{\flat}),
\end{align*}
where $L:=\graph\Omega^{\flat}, L_{\mathrm{H}}:=\graph\omega_{\mathrm{H}}^{\flat},L_{\mathrm{E}}:=\graph\omega_{\mathrm{E}}^{\flat}$ and $L_{\mathrm{P}}:=\graph\omega_{\mathrm{P}}^{\flat}$, it is sufficient to prove that $N_{\mathrm{H}}:=(\Omega^{\flat})^{-1}\circ\omega_{\mathrm{H}}^{\flat}, N_{\mathrm{E}}:=(\Omega^{\flat})^{-1}\circ\omega_{\mathrm{E}}^{\flat}$ and $N_{\mathrm{P}}:=(\Omega^{\flat})^{-1}\circ\omega_{\mathrm{P}}^{\flat}$ are Nijenhuis. By straightforward calculations, it follows that $\mathcal{T}_{N_{\mathrm{H}}}, \mathcal{T}_{N_{\mathrm{E}}}$ and $\mathcal{T}_{N_{\mathrm{P}}}$ vanish. Therefore $N_{\mathrm{H}},N_{\mathrm{E}}$ and $N_{\mathrm{P}}$ are Nijenhuis. Hence $(\Omega,\omega_{H})$ and $(\Omega,\omega_{E})$ are a $(0,1)$-symplectic pairs on $(TM\oplus \mathbb{R},(0,1))$ and $(\Omega,\omega_{P})$ is a $(0,1)$-presymplectic pair on $(TM\oplus \mathbb{R},(0,1))$.
 \end{example}
 
 Now, we show two relationships between Jacobi and Poisson pairs. 
 
 By Example \ref{Jacobi manifold}, if $\pi$ is Poisson on a Lie algebroid $A$ over $M$, then $(\pi,0)$ is Jacobi on a Jacobi algebroid $(A\oplus \mathbb{R},(0,1))$ over $M$. It is well-known that compatible Poisson structures $\pi_{1}$ and $\pi_{2}$ on a Lie algebroid $A$ induce compatible Jacobi structures $(\pi_{1},0)$ and $(\pi_{2},0)$ on a Jacobi algebroid $(A\oplus \mathbb{R},(0,1))$. The following theorem is a generalization of this relation.

\begin{theorem}
Let $(\pi_{1},\pi_{2})$ be a pair of $2$-sections on a Lie algebroid $A$ over $M$. Then $(\pi_{1},\pi_{2})$ is a Poisson pair on $A$ if and only if $((\pi_{1},0),(\pi_{2},0))$ is a Jacobi pair on a Jacobi algebroid $(A\oplus \mathbb{R},(0,1))$ over $M$.
\end{theorem}

  \begin{proof}
By Lemma \ref{2-section Nijenhuis torsion} and Lemma \ref{pi and pi' Nijenhuis torsion}, $(\pi_{1},\pi_{2})$ is a Poisson pair on $A$ if and only if $[\pi_{1},\pi_{2}]_{A,\phi_0}(\xi_{1},\xi_{2},\xi)=0$ for any $\xi_{i}$ in $\Gamma (A^{*})$ and $\xi$ in $(\pi_{1}^{\sharp})^{-1}(\mbox{Im}\,\pi_{2}^{\sharp})\cap(\pi_{2}^{\sharp})^{-1}(\mbox{Im}\,\pi_{1}^{\sharp})$, i.e., there are $\xi'$ and $\xi''$ in $\Gamma (A^{*})$ such that $\pi_{1}^{\sharp}\xi=\pi_{2}^{\sharp}\xi''$ and $\pi_{2}^{\sharp}\xi=\pi_{1}^{\sharp}\xi'$, and $((\pi_{1},0),(\pi_{2},0))$ is a Jacobi pair on a Jacobi algebroid $(A\oplus \mathbb{R},(0,1))$ if and only if $[(\pi_{1},0),(\pi_{2},0)]_{A\oplus \mathbb{R},(0,1)}((\xi_{1},f_{1}),(\xi_{2},f_{2}),(\xi,f))=0$ for any $(\xi_{i},f_{i})$ in $\Gamma (A^{*})\oplus C^{\infty}(M)$ and $(\xi,f)$ in $((\pi_{1},0)^{\sharp})^{-1}(\mbox{Im}\,(\pi_{2},0)^{\sharp})\cap((\pi_{2},0)^{\sharp})^{-1}(\mbox{Im}\,(\pi_{1},0)^{\sharp})$, i.e., there are $(\xi',f')$ and $(\xi'',f'')$ in $\Gamma (A^{*})\oplus C^{\infty}(M)$ such that $(\pi_{1},0)^{\sharp}(\xi,f)=(\pi_{2},0)^{\sharp}(\xi'',f'')$ and $(\pi_{2},0)^{\sharp}(\xi,f)=(\pi_{1},0)^{\sharp}(\xi',f')$. The condition for $(\xi,f)$ is equivalent with the conditions that $\pi_{1}^{\sharp}\xi=\pi_{2}^{\sharp}\xi''$, that $\pi_{2}^{\sharp}\xi=\pi_{1}^{\sharp}\xi'$ and that $f,f'$ and $f''$ are arbitrary. Therefore $(\xi,f)$ is in $((\pi_{1},0)^{\sharp})^{-1}(\mbox{Im}\,(\pi_{2},0)^{\sharp})\cap((\pi_{2},0)^{\sharp})^{-1}(\mbox{Im}\,(\pi_{1},0)^{\sharp})$ if and only if $\xi$ is in $(\pi_{1}^{\sharp})^{-1}(\mbox{Im}\,\pi_{2}^{\sharp})\cap(\pi_{2}^{\sharp})^{-1}(\mbox{Im}\,\pi_{1}^{\sharp})$ and $f$ is arbitrary. Then we compute
\begin{align*}
\langle (\xi ,f),[(\pi_{1},0)^{\sharp}((\xi _{1},f_{1})),(\pi_{2},0)^{\sharp}((\xi _{2},f_{2}))]_{A\oplus \mathbb{R}}\rangle&=\langle \xi,[\pi_{1}^{\sharp}\xi _{1},\pi_{2}^{\sharp}\xi _{2}]_{A}\rangle,\\
\langle (\xi ,f),[(\pi_{2},0)^{\sharp}((\xi _{1},f_{1})),(\pi_{1},0)^{\sharp}((\xi _{2},f_{2}))]_{A\oplus \mathbb{R}}\rangle&=\langle \xi,[\pi_{2}^{\sharp}\xi _{1},\pi_{1}^{\sharp}\xi _{2}]_{A}\rangle,\\
\langle (\xi ,f),(\pi_{1},0)^{\sharp}[(\xi _{1},f_{1}),(\xi _{2},f_{2})]_{(\pi_{2},0),(0,1)}\rangle&=\langle \xi,\pi_{1}^{\sharp}[\xi _{1},\xi _{2}]_{\pi_{2}}\rangle,\\
\langle (\xi ,f),(\pi_{2},0)^{\sharp}[(\xi _{1},f_{1}),(\xi _{2},f_{2})]_{(\pi_{1},0),(0,1)}\rangle&=\langle \xi,\pi_{2}^{\sharp}[\xi _{1},\xi _{2}]_{\pi_{1}}\rangle,
\end{align*}
so that by (\ref{Jacobi structure no wa}), we have
\begin{align*}
[(\pi_{1},0),(\pi_{2},0)]_{A\oplus \mathbb{R},(0,1)}((\xi _{1},f_{1}),(\xi _{2},f_{2}),(\xi ,f))=[\pi_{1},\pi_{2}]_{A}(\xi _{1},\xi _{2},\xi).
\end{align*}
Therefore the consequence holds.
\end{proof}

The other relation between Jacobi and Poisson pairs is the following theorem.

\begin{theorem}
Let $(\pi_{1},\pi_{2})$ be a pair of $2$-sections on a Jacobi algebroid $(A,\phi_{0})$ over $M$. Then $(\pi_{1},\pi_{2})$ is a Jacobi pair on $(A,\phi_{0})$ if and only if $(\tilde{\pi}_{1},\tilde{\pi}_{2})$ is a Poisson pair on a Lie algebroid $\tilde{A}_{\bar{\phi}_{0}}$ over $M\times \mathbb{R}$, where $\tilde{\pi}_{i}:=e^{-t}\pi_{i}$ in $\Gamma (\tilde{A})$. 
\end{theorem}

\begin{proof}
By Lemma \ref{Jacobi and Poisson structures}, a $2$-section $\pi$ on $A$ is a Jacobi structure on $(A,\phi_{0})$ if and only if a $2$-section $\tilde{\pi}$ on $\tilde{A}$ is a Poisson structure on $\tilde{A}_{\bar{\phi}_{0}}$. By the definitions of Jacobi and Poisson pairs and Theorem \ref{main 1}, a pair $(\pi_{1},\pi_{2})$ is a Jacobi pair on $(A,\phi_{0})$ if and only if a pair $(\tilde{\pi}_{1},\tilde{\pi}_{2})$ is a Poisson pair on $\tilde{A}_{\bar{\phi}_{0}}$.
%
\end{proof}

\subsection{J$\Omega$- and $\Omega $N-structures}

In this subsection, we define J$\Omega$- and $\Omega$N-structures on Jacobi algebroids, and show the relationship between J$\Omega$- (resp. $\Omega$N-) structures on Jacobi algebroids and P$\Omega$- (resp. $\Omega$N-) structures on Lie algebroids. By using the relationship, we show that J$\Omega$- and $\Omega$N-structures on Jacobi algebroids can be characterized in terms of Dirac pairs.

 We start with the definitions of J$\Omega$- and $\Omega$N-structures on a Jacobi algebroid.
 
\begin{defn}
 Let $(A,\phi_{0})$ be a Jacobi algebroid over $M$, $\pi$ a $2$-section on $A$ and $\omega$ a $2$-cosection on $A$. Then a pair $(\pi,\omega)$ is a {\it J$\Omega$-structure} on $(A,\phi_{0})$ if $\pi$ is Jacobi and both $\omega$ and $\omega_{N}$ are $d_{A,\phi_{0}}$-closed, where $N:=\pi^{\sharp}\circ \omega^{\flat}$ and $\omega _{N}$ is a $2$-cosection characterized by $\omega_{N}^{\flat}:=\omega^{\flat}\circ N$.  
\end{defn}

\begin{defn}
 Let $(A,\phi_{0})$ be a Jacobi algebroid over $M$, $\omega$ a $2$-cosection on $A$ and $N$ a $(1,1)$-tensor field on $A$. Then a pair $(\omega,N)$ is a {\it $\Omega$N-structure} on $(A,\phi_{0})$ if $\omega^{\flat}\circ N=N^{*}\circ \omega^{\flat}$, $N$ is Nijenhuis and both $\omega$ and $\omega_{N}$ are $d_{A,\phi_{0}}$-closed, where $\omega _{N}$ is a $2$-cosection characterized by $\omega_{N}^{\flat}:=\omega^{\flat}\circ N$. We can also define a {\it weak $\Omega$N-structure} on $(A,\phi_{0})$ by replacing $\mathcal{T}_{N}=0$ with $\omega^{\flat}(\mathcal{T}_{N}(X,Y))=0$ for any $X$ and $Y$ in $\Gamma (A)$ for the definition of a $\Omega$N-structure on $(A,\phi_{0})$.
\end{defn}

It is clear that the definitions of J$\Omega$- and (weak) $\Omega$N-structures on a Jacobi algebroid $(A,\phi_{0})$ coincide with the definitions of P$\Omega$- and (weak) $\Omega$N-structures on a Lie algebroid $A$ when $\phi_{0}=0$.

First, the following proposition means that there is one-to-one correspondence between J$\Omega$-structures on a Jacobi algebroid $(A,\phi_{0})$ and P$\Omega$-structures on a Lie algebroid $\tilde{A}_{\bar{\phi}_{0}}$.

 \begin{prop}\label{POmega and JOmega}
Let $(A,\phi_{0})$ be a Jacobi algebroid over $M$. Then a pair $(\pi,\omega)$ is a J$\Omega$-structure on $(A,\phi_{0})$ if and only if a pair $(\tilde{\pi},\tilde{\omega})$ is a P$\Omega$-structure on $\tilde{A}_{\bar{\phi}_{0}}$, where $\tilde{\pi}=e^{-t}\pi,\tilde{\omega}=e^{t}\omega$.
 \end{prop}

 \begin{proof}
 By Lemma \ref{Jacobi and Poisson structures}, a $2$-section $\pi$ on $A$ is a Jacobi structure on $(A,\phi_{0})$ if and only if a $2$-section $\tilde{\pi}$ on $\tilde{A}$ is a Poisson structure on $\tilde{A}_{\bar{\phi}_{0}}$, and a $2$-cosection $\omega$ on $A$ is a $\phi_{0}$-presymplectic structure on $(A,\phi_{0})$ if and only if a $2$-cosection $\tilde{\omega}$ on $\tilde{A}$ is a presymplectic structure on $\tilde{A}_{\bar{\phi}_{0}}$. Setting $N:=\pi^{\sharp}\circ \omega^{\flat}$, we obtain $\tilde{N}:=\tilde{\pi}^{\sharp}\circ \tilde{\omega}^{\flat}=\pi^{\sharp}\circ \omega^{\flat}=N$. For any $\tilde{X}$ and $\tilde{Y}$ in $\Gamma (\tilde{A})$,
\begin{align*}
(\omega_{N})^{\mbox{\textasciitilde}}(\tilde{X},\tilde{Y})=e^{t}\omega_{N}(\tilde{X},\tilde{Y})=e^{t}\langle\omega^{\flat}N\tilde{X},\tilde{Y}\rangle=\langle\tilde{\omega}^{\flat}\tilde{N}\tilde{X},\tilde{Y}\rangle=\tilde{\omega}_{\tilde{N}}(\tilde{X},\tilde{Y}).
\end{align*}
Hence $(\omega_{N})^{\mbox{\textasciitilde}}=\tilde{\omega}_{\tilde{N}}$. Then since $\bar{d}_{A}^{\phi_{0}}\tilde{\omega}_{\tilde{N}}=\bar{d}_{A}^{\phi_{0}}(\omega_{N})^{\mbox{\textasciitilde}}=e^{t}d_{A,\phi_{0}}\omega_{N}$ by Lemma \ref{Jacobi and Poisson structures}, it follows that $\omega_{N}$ is $d_{A,\phi_{0}}$-closed if and only if $\tilde{\omega}_{\tilde{N}}$ is $\bar{d}_{A}^{\phi_{0}}$-closed.
\end{proof}

 \begin{prop}\label{OmegaN and OmegaN}
Let $(A,\phi_{0})$ be a Jacobi algebroid over $M$. Then a pair $(\omega,N)$ is an $\Omega$N- (resp. a weak $\Omega$N-)structure on $(A,\phi_{0})$ if and only if a pair $(\tilde{\omega},N)$ is an $\Omega$N- (resp. a weak $\Omega$N-)structure on $\tilde{A}_{\bar{\phi}_{0}}$, where $\tilde{\pi}=e^{-t}\pi,\tilde{\omega}=e^{t}\omega$ and a $(1,1)$-tensor field $N$ on $A$ is regarded as a $(1,1)$-tensor field independent of $t$ on $\tilde{A}_{\bar{\phi}_{0}}$.
 \end{prop}

 \begin{proof}
By the proof of Proposition \ref{POmega and JOmega}, a $2$-cosection $\omega$ on $A$ is a $\phi_{0}$-presymplectic structure on $(A,\phi_{0})$ if and only if a $2$-cosection $\tilde{\omega}$ on $\tilde{A}$ is a presymplectic structure on $\tilde{A}_{\bar{\phi}_{0}}$. We have
\begin{align*}
\tilde{\omega}^{\flat}\circ N=e^{t}\omega^{\flat}\circ N,\quad 
N^{*}\circ\tilde{\omega}^{\flat}=e^{t}N^{*}\circ\omega^{\flat},
\end{align*}
so that the commutativity of $\tilde{\omega}$ and $N$ is equivalent with the commutativity of $\omega$ and $N$. By the proof of Proposition \ref{POmega and JOmega} again, $\omega_{N}$ is $d_{A,\phi_{0}}$-closed if and only if $\tilde{\omega}_{\tilde{N}}$ is $\bar{d}_{A}^{\phi_{0}}$-closed. Finally, by Lemma \ref{Nijenhuis structure property}, since $\mathcal{T}_{N}^{A}(\tilde{X},\tilde{Y})=\mathcal{T}_{N}^{\tilde{A}_{\bar{\phi}_{0}}}(\tilde{X},\tilde{Y})$ for any $\tilde{X}$ and $\tilde{Y}$ in $\Gamma (\tilde{A})$
, the consequence holds.
\end{proof}

For a weak $\Omega$N-structure on $(A,\phi_{0})$, we set
\begin{align*}
\mathrm{N}_{(\omega,N)}^{+}:=\{(\omega^{\flat}X,\omega^{\flat}_{N}X)\,|\, X\in A\}\subset \mathrm{N}_{L,L'}^{*},
\end{align*}
where $L:=\graph\omega^{\flat}$ and $L':=\graph\omega_{N}^{\flat}$. The following theorem is characterizations of J$\Omega$- and (weak) $\Omega$N-structures on a Jacobi algebroid $(A,\phi_{0})$ by Dirac pairs, and generalizations of Proposition \ref{POmega to Dirac} and Proposition \ref{weak OmegaN to Dirac}.

 \begin{theorem}\label{JOmega and OmegaN to Dirac}
 Let $(A,\phi_{0})$ be a Jacobi algebroid over $M$, $\pi$ a Jacobi structure on $(A,\phi_{0})$, $\omega$ a $\phi_{0}$-presymplectic structure on $(A,\phi_{0})$ and $N$ a $(1,1)$-tensor field on $A$. Then
\begin{enumerate}
\item[{\rm (i)}] If a pair $(\pi,\omega )$ is a J$\Omega$-structure on $(A,\phi_{0})$, then a pair $(\overline{\graph \pi^{\sharp}}$, $\graph\omega^{\flat})$ is a Dirac pair on $(A,\phi_{0})$. Conversely, if $(\overline{\graph \pi^{\sharp}},\graph\omega^{\flat})$ is a Dirac pair on $(A,\phi_{0})$, and if $\pi $ is non-degenerate, then a pair $(\pi,\omega )$ is a J$\Omega$-structure on $(A,\phi_{0})$.
\item[{\rm (ii)}] If a pair $(\omega ,N)$ is an $\Omega$N-structure on $(A,\phi_{0})$, and if $\mathrm{N}_{(\omega,N)}^{+}=\mathrm{N}_{L,L'}^{*}$, where $L:=\graph\omega^{\flat}$ and $L':=\graph\omega_{N}^{\flat}$, then a pair $(L,L')$ is a Dirac pair on $(A,\phi_{0})$. Conversely, if $(\graph \omega^{\flat},\graph\omega_{N}^{\flat})$ is a Dirac pair on $(A,\phi_{0})$, then a pair $(\omega ,N)$ is a weak $\Omega$N-structure on $(A,\phi_{0})$.
\end{enumerate}
 \end{theorem}

\begin{proof}
(i) By Proposition \ref{POmega and JOmega}, $(\pi,\omega)$ is a J$\Omega$-structure on $(A,\phi_{0})$ if and only if $(\tilde{\pi},\tilde{\omega})$ is a P$\Omega$-structure on $\tilde{A}_{\bar{\phi}_{0}}$. By (i) in Proposition \ref{POmega to Dirac}, $(\overline{\graph \tilde{\pi}^{\sharp}},\graph\tilde{\omega}^{\flat})$ is a Dirac pair on $\tilde{A}_{\bar{\phi}_{0}}$. By (ii) in Theorem \ref{main 1}, $(\overline{\graph\tilde{\pi}^{\sharp}},\graph\tilde{\omega}^{\flat})$ is a Dirac pair on $\tilde{A}_{\bar{\phi}_{0}}$ if and only if $(\overline{\graph\pi^{\sharp}},$ $\graph\omega^{\flat})$ is a Dirac pair on $(A,\phi_{0})$. The converse holds by (ii) in Theorem \ref{main 1}, (ii) in Proposition \ref{POmega to Dirac} and Proposition \ref{POmega and JOmega}.

\noindent
(ii) By Proposition \ref{OmegaN and OmegaN}, $(\omega,N)$ is an $\Omega$N-structure on $(A,\phi_{0})$ if and only if $(\tilde{\omega},N)$ is an $\Omega$N-structure on $\tilde{A}_{\bar{\phi}_{0}}$. For any $(\tilde{\omega}^{\flat}\tilde{X},\tilde{\omega}_{N}^{\flat}\tilde{X})$ in $\mathrm{N}_{(\tilde{\omega},N)}^{+}$, we have
\begin{align*}
(\tilde{\omega}^{\flat}\tilde{X},\tilde{\omega}_{N}^{\flat}\tilde{X})=(\omega^{\flat}(e^{t}\tilde{X}),\omega_{N}^{\flat}(e^{t}\tilde{X})).
\end{align*}
Since if $\tilde{X}$ in $\Gamma (\tilde{A})$ is arbitrary, then so is $e^{t}\tilde{X}$ in $\Gamma (\tilde{A})$, it follows
\begin{align*}
\mathrm{N}_{(\tilde{\omega},N)}^{+}=\{(\omega^{\flat}\tilde{X},\omega_{N}^{\flat}\tilde{X})\,|\,\tilde{X} \in \Gamma (\tilde{A})\}.
\end{align*}
Therefore $\mathrm{N}_{(\tilde{\omega},N)}^{+}$ is the set of all curves in $\mathrm{N}_{(\omega,N)}^{+}$. On the other hand, for any $(X,Y)$ in $\mathrm{N}_{L,L'}$, it follows $\omega^{\flat}X=\omega_{N}^{\flat}Y$. Since this is equivalent with $\tilde{\omega}^{\flat}X=\tilde{\omega}_{N}^{\flat}Y$, we obtain $\mathrm{N}_{L,L'}\subset \mathrm{N}_{\tilde{L},\tilde{L}'}$, where $\tilde{L}:=\graph\tilde{\omega}^{\flat},\tilde{L}':=\graph\tilde{\omega}_{N}^{\flat}$. For any $(\tilde{\beta},\tilde{\alpha })$ in $\mathrm{N}_{\tilde{L},\tilde{L}'}^{*}$ and $(X,Y)$ in $\mathrm{N}_{L,L'}\subset \mathrm{N}_{\tilde{L},\tilde{L}'}$, since $\langle \tilde{\alpha},X\rangle=\langle\tilde{\beta},Y\rangle$, $\mathrm{N}_{\tilde{L},\tilde{L}'}^{*}$ is the set of all curves in $\mathrm{N}_{L,L'}^{*}$. Since $\mathrm{N}_{L,L'}^{*}=\mathrm{N}_{(\omega,N)}^{+}$, we obtain $\mathrm{N}_{\tilde{L},\tilde{L}'}^{*}=\mathrm{N}_{(\tilde{\omega},N)}^{+}$. By (i) in Proposition \ref{weak OmegaN to Dirac}, a pair $(\tilde{L},\tilde{L}')$ is a Dirac pair on $\tilde{A}_{\bar{\phi}_{0}}$. By (iii) in Theorem \ref{main 1}, $(\tilde{L},\tilde{L}')$ is a Dirac pair on $\tilde{A}_{\bar{\phi}_{0}}$ if and only if $(\graph\omega^{\flat},\graph\omega_{N}^{\flat})$ is a Dirac pair on $(A,\phi_{0})$. The converse holds by (iii) in Theorem \ref{main 1}, (ii) in Proposition \ref{weak OmegaN to Dirac} and Proposition \ref{OmegaN and OmegaN}.
\end{proof}

\begin{rem}
Theorem \ref{JOmega and OmegaN to Dirac} can also be proved directly by long calculations. However, as above, we can prove Theorem \ref{JOmega and OmegaN to Dirac} more easily by using Theorem \ref{main 1}, Proposition \ref{POmega and JOmega}, Proposition \ref{OmegaN and OmegaN} and the theory of Dirac pairs on Lie algebroids.  
\end{rem}


\begin{example}
In Example \ref{cotan r2 times r}, we denote the opposite of the non-degenerate Jacobi structure corresponding with a $(0,1)$-presymplectic structure $\Omega$ on $(TM\oplus \mathbb{R},(0,1))$ by $\Pi$, i.e., $\Pi$ is a $2$-vector field characterized by $\Pi^\sharp=(\Omega^\flat)^{-1}$. Then three pairs $(\Pi,\omega_{\mathrm{H}})$, $(\Pi,\omega_{\mathrm{E}})$ and $(\Pi,\omega_{\mathrm{P}})$ are J$\Omega$-structures on $(TM\oplus \mathbb{R},(0,1))$ and three pairs $(\Omega,N_{\mathrm{H}})$, $(\Omega,N_{\mathrm{E}})$ and $(\Omega,N_{\mathrm{P}})$ are $\Omega$N-structures on $(TM\oplus \mathbb{R},(0,1))$. In fact, 
we have
\begin{align*}
(\overline{\graph \Pi^{\sharp}},\graph\omega_{\mathrm{H}}^{\flat})&=(\graph (\Pi^{\sharp})^{-1},\graph\omega_{\mathrm{H}}^{\flat})\\
         &=(\graph\Omega^{\flat},\graph\omega_{\mathrm{H}}^{\flat}),\\
(\graph \Omega^{\flat},\graph\Omega_{N_\mathrm{H}}^{\flat})&=(\graph \Omega^{\flat},\graph(\Omega^{\flat}\circ N_{\mathrm{H}}))\\
         &=(\graph \Omega^{\flat},\graph(\Omega^{\flat}\circ (\Omega^{\flat})^{-1}\circ\omega_{\mathrm{H}}^\flat))\\
         &=(\graph \Omega^{\flat},\graph\omega_{\mathrm{H}}^\flat).
\end{align*}
Therefore, since $(\Omega,\omega_{\mathrm{H}})$ is a $\phi_0$-presymplectic pair by Example \ref{cotan r2 times r}, $(\graph \Omega^{\flat},$ $\graph\omega_{\mathrm{H}}^\flat)=(\overline{\graph \Pi^{\sharp}},\graph\omega_{\mathrm{H}}^{\flat})=(\graph \Omega^{\flat},\graph\Omega_{N_\mathrm{H}}^{\flat})$ is Dirac.
Since $\Omega$ is non-degenerate, $(\Pi,\omega_{\mathrm{H}})$ is a J$\Omega$-structure on $(TM\oplus \mathbb{R},(0,1))$ by (i) in Theorem \ref{JOmega and OmegaN to Dirac}. It also follows similarly that $(\Pi,\omega_{\mathrm{E}})$ and $(\Pi,\omega_{\mathrm{P}})$ are J$\Omega$-structures on $(TM\oplus \mathbb{R},(0,1))$. On the other hand, by (ii) in Theorem \ref{JOmega and OmegaN to Dirac}, $(\Omega,N_\mathrm{H})$ is a weak $\Omega$N-structure on $(TM\oplus \mathbb{R},(0,1))$. Since $N_\mathrm{H}$ is Nijenhuis, $(\Omega,N_\mathrm{H})$ is an $\Omega$N-structure on $(TM\oplus \mathbb{R},(0,1))$. It also follows similarly that $(\Omega,N_{\mathrm{E}})$ and $(\Omega,N_{\mathrm{P}})$ are $\Omega$N-structures on $(TM\oplus \mathbb{R},(0,1))$.
\end{example}

\end{document}